\documentclass[12pt,a4paper]{amsart}
\usepackage{amsmath, amssymb, amsfonts, amsthm, float, stmaryrd, epsfig}
\usepackage[abbrev,alphabetic]{amsrefs}
\usepackage[pdfborder={0 0 0 [0 0 ]},bookmarksdepth=3, bookmarksopen=true]{hyperref}
\usepackage[latin1,utf8]{inputenc}
\usepackage[T1]{fontenc}
\usepackage{tikz-cd}
\usepackage[minimal]{yhmath}
\usetikzlibrary{arrows}
\usepackage{adjustbox}

\usepackage[capitalize]{cleveref}
\usepackage{color}
\usepackage[ps,all,arc,rotate]{xy}
\usepackage{bbm} 

\BibSpecAlias{misc}{webpage}

\usepackage{stackengine,scalerel}
\stackMath

\usepackage{booktabs}

\hypersetup{
    colorlinks=true,
    linkcolor=black,
    citecolor=black,
    filecolor=black,
    urlcolor=black,
}

\numberwithin{figure}{section}
\numberwithin{table}{section}

\theoremstyle{plain}
\newtheorem{thm}{Theorem}[section]
\crefname{thm}{Theorem}{Theorems}
\newtheorem*{prop*}{Proposition}
\newtheorem*{thm*}{Theorem}
\newtheorem{prop}[thm]{Proposition}
\crefname{prop}{Proposition}{Propositions}
\newtheorem{lem}[thm]{Lemma}
\crefname{lem}{Lemma}{Lemmata}
\newtheorem{cor}[thm]{Corollary}
\crefname{cor}{Corollary}{Corollaries}

\crefname{conj}{Conjecture}{Conjectures}
\crefname{equation}{Equation}{Equations}

\newtheorem{thmx}{Theorem}
\crefname{thmx}{Theorem}{Theorems}
\newtheorem{corx}[thmx]{Corollary}

\theoremstyle{definition}

\newtheorem{dfn}[thm]{Definition}
\newtheorem*{dfn*}{Definition}

\theoremstyle{remark}
\newtheorem{rmk}[thm]{Remark}

\newtheoremstyle{maintheorem}{}{}{\itshape}{}{\bfseries}{}{.5em}{#1 \!\thmnote{\ #3}.}
\theoremstyle{maintheorem}
\newtheorem*{mainthm}{Theorem}
\newtheorem*{maincor}{Corollary}

\makeatletter
\let\c@figure\c@thm
\let\c@table\c@thm
\makeatother
\crefname{figure}{Figure}{Figures}
\crefname{table}{Table}{Tables}

\newcommand{\id}{\operatorname{id}}
\newcommand{\im}{\operatorname{im}}
\newcommand{\supp}{\operatorname{supp}}

\def\R{\mathbb{R}}
\def\N{\mathbb{N}}
\def\Z{\mathbb{Z}}
\def\Q{\mathbb{Q}}

\def\1{\mathbbm{1}}

\def\K{\mathbb{K}}
\def\F{\mathbb{F}}

\DeclareMathOperator{\homol}{H}
\DeclareMathOperator{\Hom}{Hom}
\DeclareMathOperator{\PD}{PD}

\newcommand{\D}{\mathcal{D}}

\def\s-{\smallsetminus}

\def\iff{if and only if }

\newcommand{\nov}[3]{{\widehat{#1 #2}^{#3}}}
\newcommand{\novq}[2]{{\widehat{\Q #1}^{#2}}}

\newcommand{\betti}[3]{\beta_{#2}^{(2)}(#1; #3)}
\newcommand{\bettiQ}[2]{\beta_{#2}^{(2)}(#1)}

\newcommand{\typeFP}{\mathtt{FP}}

\binoppenalty=\maxdimen
\relpenalty=\maxdimen

\newcounter{dawidcomments}

\newcounter{giovannicomments}

\newcounter{samcomments}

\author{Sam P.\ Fisher}

\author{Giovanni Italiano}

\author{Dawid Kielak}
\email{fisher@maths.ox.ac.uk}
\email{italiano@maths.ox.ac.uk}
\email{kielak@maths.ox.ac.uk}
\address{Mathematical Institute,
	Andrew Wiles Building,
	Observatory Quarter,
	University of Oxford,
	Oxford,
	OX2 6GG,
	United Kingdom}

\title{Virtual fibring of Poincar\'e-duality groups}

\dedicatory{Dedicated to Martin R.\ Bridson, on his $60$th birthday}

\begin{document}

\begin{abstract}
    We show that a RFRS Poincar\'e-duality group $G$ admits a virtual epimorphism to the integers whose kernel is itself a Poincar\'e-duality group over every  field if and only if the $L^2$-homology of $G$ vanishes and so do the positive-characteristic variants thereof.
    
    Our investigations yield a more general relationship between cohomology at infinity of groups that algebraically fibre and their fibres. In particular, we show that if the fundamental group of  an aspherical manifold of dimension at least three algebraically fibres, then the fibre is one ended.
\end{abstract}

\maketitle

\section{Introduction}

A smooth manifold that fibres over the circle can be understood not only as a static object, but also as a manifold of codimension one, the fibre, changing over time, with the dynamics controlled by the monodromy. This description is indeed helpful in many cases, for example in three-manifolds, where our understanding of the fibres and their homeomorphisms is very detailed.

If a compact connected smooth manifold $M$ fibres over the circle with fibre $N$ (which then is also a manifold), on the level of fundamental groups there is a short exact sequence
\[
\mathbbm{1} \to \pi_1(N) \to \pi_1(M) \to \Z \to \mathbbm{1}.
\]
In particular, we get a surjective homomorphism $\pi_1(M) \to \Z$ whose kernel $K=\pi_1(N)$ is the fundamental group of a compact manifold. If $M$ is additionally aspherical, then so is $N$ and we deduce that the kernel $K$ is a group admitting a compact classifying space (a group of type $\mathtt{F}$), and hence the cellular chain complex of the universal cover of the classifying space is a finite resolution by finitely generated projective $\Z K$-modules of the trivial module $\Z$ (that is, $K$ is of type $\typeFP$). Hence, admitting a fibration over the circle is reflected in the algebraic structure of $\pi_1(M)$.

This statement has a strong converse in the realm of three-manifolds: as proved by Stallings \cite{Stallings1962},  if the fundamental group of a compact irreducible three-manifold maps onto the integers with a finitely generated kernel that is not $\Z /2\Z$, then the three-manifold fibres over the circle in a way inducing the given epimorphism. The condition on the kernel not being $\Z /2\Z$ can now be dropped thanks to the resolution of the Poincar\'e Conjecture.

In higher dimensions, there is a result of Farrell \cite{Farrell1967}: if $M$ is a closed aspherical smooth manifold of dimension at least six, then an epimorphism $\pi_1(M) \to \Z$ is induced by a fibration \iff its kernel admits a model for its classifying space that is finitely dominated, and a certain $K$-theoretic obstruction vanishes. The first condition can be rephrased by asking for the kernel to be finitely presented and of type $\typeFP$.

The situation changes if instead of asking whether a specific epimorphism to the integers comes from fibring, we ask if such an epimorphism exists at all.
Going back to three-manifolds, a central theme in their study that emerged due to the groundbreaking insights of Thurston was that it is often interesting not only to understand when a given manifold fibres over the circle, but also when it virtually fibres, that is, admits a finite-sheeted covering that fibres. Thurston conjectured that this should be the case for all finite-volume hyperbolic three-manifolds, and this conjecture was confirmed by Wise \cite{Wise2012}  in the cusped case and  Agol \cite{Agol2013} in the closed case. 

Both proofs consist of two main steps. The first step shows that fundamental groups of hyperbolic three-manifolds virtually (i.e. up to finite index) have the RFRS property.
For a finitely generated group, the RFRS condition is equivalent to being residually \{virtually abelian and locally indicable\}, and so it is a strong residual property. To establish that a group is virtually RFRS, one usually goes via the Haglund--Wise theory of special cube complexes \cite{HaglundWise2008}, and this is indeed how Wise and Agol proceed, in the cusped case using in fundamental ways the work of Culler--Shalen~\cite{CullerShalen1984}, and in the closed case that of Kahn--Markovi\'c \cite{KahnMarkovic2012} and Bergeron--Wise \cite{BergeronWise2012}.

The second step of the proofs is an application of another theorem of Agol~\cite{Agol2008}, stating that a compact irreducible three-manifold with RFRS fundamental group virtually fibres \iff its Euler characteristic is zero. There are many ways of computing the Euler characteristic; for our purpose, we will focus on its being the alternating sum of Betti numbers, and thus a homological invariant. Hence, Agol established that for RFRS three-manifold groups, virtual fibring is controlled by a homological invariant.

This statement has  algebraic counterparts: first it was shown in \cite{Kielak2020a} that if $G$ is an infinite RFRS group, then $G$ admits a virtual epimorphism to the integers with a finitely generated kernel \iff the first $L^2$-Betti number of $G$ vanishes. Soon afterwards, Jaikin-Zapirain \cite{Jaikin-Zapirain2021}*{Appendix} introduced the theory of $L^2$-Betti numbers over ground fields other than $\Q$ for RFRS groups. This was then used by the first author, who 
in \cite{Fisher2024} showed that $G$ admits a virtual epimorphism to $\Z$ with kernel of type $\typeFP_n(\K)$ \iff the $L^2$-Betti numbers of $G$ over $\K$ vanish up to degree $n$. 
 We will introduce all of these notions in the main text -- the key point for now is that, as in Agol's case, virtual algebraic fibring is controlled by homological invariants, namely various versions of $L^2$-homology.

The question motivating our investigations here is: do homological invariants control virtual fibring of compact aspherical manifolds with RFRS fundamental groups also beyond dimension three? In view of Farrell's result, one has to find ways of controlling finiteness properties of kernels of virtual epimorphisms to $\Z$, and the $K$-theoretic obstruction. Thanks to the work of Siebenmann \cite{Siebenmann1970}, we see that this obstruction lives in the Whitehead group of the manifold, and this group is known to vanish in many cases of interest for us, most notably when the fundamental group of the manifold is compact special in the sense of Haglund--Wise -- such groups are $\mathrm{CAT}(0)$, and hence satisfy the Farrell--Jones Conjecture, as proved by Bartels--L\"uck \cite{BartelsLueck2012}. Therefore, it is reasonable to focus first  on securing the relevant finiteness properties.

Establishing finite presentability seems to be out of reach of current methods. We therefore focus on type $\typeFP$. Since this is a solely homological property, it would be too restrictive to stay within the realm of compact aspherical manifolds. A more natural setting is provided by Poincar\'e-duality groups -- they are of type $\typeFP$, and their homology is related to their cohomology precisely in the same way as in the case of compact aspherical manifolds. More generally, given a field $\K$ we can talk about Poincar\'e duality over $\K$, where the group is of type $\typeFP(\K)$ and the duality between homology and cohomology only applies to $\K G$-modules as coefficients. In this generality, we offer the following result.

\begin{thmx}\label{main pd single field}
	Let $\K$ be a field and $G$ be a RFRS orientable  Poincar\'e-duality group of dimension $n>0$ over $\K$. The following are equivalent:
	\begin{enumerate}
		\item There exists a finite-index subgroup $G_0$ in $G$ and an epimorphism $\phi \colon G_0 \to \Z$ such that $\ker \phi$ is an orientable  Poincar\'e-duality group of dimension $n-1$ over $\K$;
		\item For every $i \leqslant n$ we have $\betti{G}{i}{\K} = 0$.
	\end{enumerate}
\end{thmx} 

In fact, we can also deal with all fields simultaneously, but this requires a stronger assumption on $G$. 

\begin{thmx}\label{main pd}
	Let $G$ be a RFRS orientable  Poincar\'e-duality group of dimension $n>0$. The following are equivalent:
	\begin{enumerate}
		\item There exists a finite-index subgroup $G_0$ in $G$ and an epimorphism $\phi \colon G_0 \to \Z$ such that $\ker \phi$ is an orientable  Poincar\'e-duality group of dimension $n-1$ over all fields;
		\item For all $i \leqslant n$ and all prime fields $\K$, we have $\betti{G}{i}{\K} = 0$.
	\end{enumerate}
\end{thmx}

Using coincidences of small dimensions, we obtain a sharper result in dimension four.
	
\begin{corx}
	\label{main dim 4}
	Let $G$ be a RFRS Poincar\'e-duality group of dimension four. The following are equivalent:
	\begin{enumerate}
		\item There exists a finite-index subgroup $G_0 \leqslant G$ and an epimorphism $\phi \colon G_0 \rightarrow \Z$ such that $\ker \phi$ is an orientable Poincar\'e-duality group of dimension three;
		\item We have $\betti{G}{1}{\Q}=\betti{G}{2}{\Q} = 0$.
	\end{enumerate}
\end{corx}
The Betti numbers $\betti{G}{i}{\Q}$ appearing above are the usual $L^2$-Betti numbers $\bettiQ{G}{i}$.

	\smallskip
	
	In the process of proving the results above, we establish a relationship between cohomology at infinity of any group $G$ and algebraic fibres sitting within.
	
	\begin{thmx}
		\label{main at infinity}
Let $R$ be a ring. Let $G$ be a group with an epimorphism $\phi \colon G \to \Z$. If  $K = \ker \phi$ is of type $\typeFP_n(R)$ then 
\[
\homol^{i+1}(G;RG) \cong \homol^i(K; RK)
\]
as $RK$-modules for all $i<n$  and $\homol^n(K;RK)$ is isomorphic as an $RK$-module to a submodule of  $\homol^{n+1}(G;RG)$.
	\end{thmx}
	
Applying the theorem with $i=1$ yields the following.

	\begin{corx}
		\label{one-ended fibr}
	Let $M$ be a closed aspherical manifold of dimension at least three. If $\phi \colon \pi_1(M) \to \Z$ is an epimorphism with finitely generated kernel, then the kernel has only one end. 
	\end{corx}
	
\subsection*{Outline of the arguments}

The key technical tool is Novikov homology, that is group homology with coefficients in a Novikov ring. If we have a homomorphism $\phi \colon G \to \Z$ and a ring $R$, the Novikov ring $\nov R G \phi$ is the ring of twisted Laurent power series in $t$ with coefficients in $\ker \phi$, where $t \in G$ is such that $\phi(t)$ is the positive generator of $\Z$, and the twisting comes from the conjugation action of $t$ on $\ker \phi$. Novikov homology controls the homological finiteness properties of $\ker \phi$.  

The first observation we make is that vanishing of Novikov homology also gives vanishing of Novikov cohomology (\cref{flip}). This has already been used in a more restricted form in \cite{Bridsonetal2025}. Then we formulate a long exact sequence relating the Novikov cohomology of $G$, and the cohomology at infinity of $G$ and $\ker \phi$ (\cref{les}). This is essentially enough to prove \cref{main at infinity}. A similar long exact sequence appeared already in \cite{Fisher2024_cdKernels}.

The final technical novelty is \cref{prop:large_primes}, in which we show that having vanishing Novikov homology over $\Q$ gives analogous vanishing for finite prime fields $\F_p$, except possibly for finitely many primes.

We then turn our attention to a RFRS group $G$. We fix a collection of fields $\mathcal K$, and prove that vanishing of $\betti{G}{i}{\K}$ for all $i \leqslant n$ and all fields $\K \in \mathcal K$ is equivalent to virtual algebraic fibring with kernel of type $\typeFP_n(\K)$ for every $\K \in \mathcal K$ simultaneously (\cref{l2 all chars}). This holds when $\mathcal K$ satisfies a certain property -- the two key cases are that of a single field, or all prime fields. In the latter case, this relies on \cref{prop:large_primes}; in both cases we need the virtual algebraic fibring theorems of the first and third authors.

Finally, for RFRS Poincar\'e-duality groups, we combine \cref{l2 all chars}, the vanishing of Novikov cohomology, and \cref{main at infinity}, and obtain \cref{main pd} as a result.

The entire paper is written in the generality of group pairs $(G, \mathcal H)$. This makes the arguments perhaps a little less transparent, but is natural, since for both aspherical manifolds and Poincar\'e-duality groups one often wants to argue by using manifolds with boundary or Poincar\'e-duality pairs. 

\subsection*{Acknowledgements}
The first author is grateful to Andrei Jaikin-Zapirain for a useful correspondence and is supported by the National Science and Engineering Research Council of Canada (ref.~no.~567804-2022). 

The second author gratefully acknowledges support from the Royal Society through the Newton International Fellowship (award number: NIF\textbackslash R1\textbackslash 231857).

The third author is grateful to Sami Douba for raising the question of one-endedness of algebraic fibres.

This work has received funding from the European Research Council (ERC) under the European Union's Horizon 2020 research and innovation programme (Grant agreement No. 850930).

For the purpose of Open Access, the authors have applied a CC BY public copyright licence to any Author Accepted Manuscript (AAM) version arising from this submission.

\section{Preliminaries}

\subsection{Modules}
\label{sec modules}
Throughout, all rings will be associative, unital, and non-zero. The symbol $R$ will always denote such a ring.

 In \cref{sec PD,sec PD fibring} we will be more restrictive, and will only look at commutative rings. In fact, the true focus of the paper is the ring $\Z$ and various fields.

We will have a preference for left modules; in particular, resolutions will be by left modules. This means that the natural modules to consider as coefficients in cohomology will be left modules, but in homology one should use right modules. It is however customary to follow  \cite{Brown1982}, and to put left modules as coefficients in homology as well: when computing homology with coefficients in a left module $M$, we tensor a resolution on the left with $M^*$, the right module obtained from $M$ by twisting the group action of $G$ by the anti-automorphism $g \mapsto g^{-1}$. When $M$ is a bimodule, its left-module structure is used in the computation of homology, but its right-module structure remains -- after twisting, it turns into a left-module structure and descends to a left-module structure for the homology.

\subsection{Characters}

Let $G$ be a finitely generated group. The \emph{character cone} of $G$ is the set $S(G)$ of non-zero characters $\phi \colon G \rightarrow \R$; alternatively, 
\[
S(G) = \homol^1(G;\R) \smallsetminus \{0\}.
\]
Then $S(G)$ inherits a topology by viewing it as a subspace of 
\[\homol^1(G;\R) \cong \R^{\beta_1(G)}.
\]
The elements of $S(G)$ will be called \emph{characters}.

When $G_0$ is a finite-index subgroup of $G$ then $S(G)$ is naturally a subspace of $S(G_0)$. For a general subgroup $H \leqslant G$ we have the induced continuous map $\homol^1(G;\R) \to \homol^1(H;\R)$, which we may restrict to $S(G)$, but its image is not guaranteed to lie in $S(H)$.

\subsection{Novikov rings}
\label{section Novikov}

Given a group $G$ we denote the group ring of $G$ with coefficients in $R$ by $RG$. 
We will identify $RG$ with the abelian group of functions $G \to R$ of finite support, where the \emph{support} $\supp x$ of $x \colon G \to R$ is $x^{-1}(R \smallsetminus \{0\})$. Multiplication is then given by convolution.

Given a homomorphism $\phi \colon G \to \R$, we define the \emph{Novikov ring} of $G$ with respect to $\phi$, denoted $\nov R G \phi$, to be the abelian group of those functions $G \to R$ whose support intersected with $\phi^{-1}((-\infty, \kappa])$ is finite for every $\kappa \in \R$. Convolution endows this abelian group with a ring structure.

The Novikov ring $\nov R G \phi$ contains $R G$, and hence is naturally an $RG$-bimodule.

\smallskip
Given an element $x \colon G \to R$ of the Novikov ring $\nov R G \phi$ and a constant $\kappa \in \R$, we define the \emph{truncation} of $x$ at $\kappa$ to be the function $G \to R$ with
\[
g \mapsto \left\{ \begin{array}{cl} x(g) & \textrm{ if } \phi(g) \leqslant \kappa, \\ 0 & \textrm{ otherwise.} \end{array} \right.
\] 
It follows immediately from the definition of the Novikov ring  that such a truncation is actually a function of finite support, and hence an element of the group ring $R G$.

Given $x \in \nov R G \phi$, we say that it has \emph{positive support} if 
\[\phi(\supp x) \subseteq (0,\infty).\]
 For every $x$, the truncation $\overline x$ of $x$ at $0$ has the property that $x - \overline x$ has positive support.

\smallskip

There is an alternative way of viewing Novikov rings when $\im(\phi) \cong \Z$: take $t \in G$ such that $\phi(t)$ is the positive generator of $\im(\phi)$, with the ordering inherited from $\im(\phi)<\R$. As an $R (\ker \phi)$-module, the Novikov ring $\nov R G \phi$ is isomorphic to the ring of twisted Laurent series over $R (\ker \phi)$ with variable $t$, with twisting induced by conjugation.

Either way, we see that as an $R \ker \phi$-module, the Novikov ring  $\nov R G \phi$ embeds into the module $\prod_{i \in \Z} ( R \ker \phi) t^i$, the module of all functions $G \to R$ whose support intersected with $\phi^{-1}([-\kappa, \kappa])$ is finite for every $\kappa \in \R$. This latter module is easily seen to be an $RG$-bimodule.

If $H \leqslant G$ is a finite-index subgroup, it is easy to see that $\nov R G \phi$ coincides as an $RG$-module with the $RG$-module induced from the $RH$-module $
\nov R H {\phi|_{H}}$.

For general subgroups, we have the following.

\begin{lem}\label{lem:restrict_novikov}
	Let $G$ be a group, let $R$ be a ring, and let $H \leqslant G$ be a subgroup. If $\homol_n(H;\nov{R}{G}{\phi}) = 0$ for some $\phi \colon G \rightarrow \R$ and $n \geqslant 0$, then $\homol_n(H;\nov{R}{H}{\phi|_H}) = 0$.
\end{lem}
\begin{proof}
	Since elements of $\nov{R}{G}{\phi}$ are functions $G \rightarrow R$, we can restrict them to $H$ to obtain an $RH$-module homomorphism $\nov{R}{G}{\phi} \rightarrow \nov{R}{H}{\phi|_H}$. The natural inclusion $\nov{R}{H}{\phi|_H} \hookrightarrow \nov{R}{G}{\phi}$ splits this homomorphism, which shows that $\nov{R}{H}{\phi|_H}$ is a direct summand of $\nov{R}{G}{\phi}$; the result then follows at once. \qedhere
\end{proof}

\subsection{Finiteness properties}

We will follow the standard (that is, Bieri's \cite{Bieri1981}) notation for finiteness properties. In particular, a group $G$ will be \emph{of type $\typeFP_n(R)$} if the trivial $RG$-module $R$ admits a resolution by left projective $RG$-modules
\[
\dots \rightarrow P_1 \rightarrow P_0 \rightarrow R \rightarrow 0
\]
with $P_i$ finitely generated for every $i \leqslant n$ (we allow $n=\infty$ here); the group will be \emph{of type $\typeFP(R)$} if there is a finite resolution of $R$ by finitely generated projective modules.

If the ring $R$ is not specified, then we are taking the respective property over $\Z$; this is reasonable, since being of type $\typeFP_n(\Z)$ implies being of type $\typeFP_n(R)$ for every ring $R$. More generally, if $S$ is an $R$-algebra, then being of type $\typeFP_n(R)$ implies being of type $\typeFP_n(S)$.

All of the finiteness properties listed here pass to subgroups of finite index.

\smallskip

We say that a group $G$ has \emph{cohomological dimension} $n$ over $R$ if the trivial $RG$-module $R$ admits a projective resolution of length $n$, but does not admit such a resolution of length $n-1$. We write $\mathrm{cd}_R(G) = n$.

We always have $\mathrm{cd}_R(G) \leqslant n$ if $\homol^i(G;M) = 0$ for every $RG$-module $M$ and all $i > n$ -- this follows from \cite{Brown1982}*{Sections III.7 and VIII.2}. In fact, an alternative definition of $\mathrm{cd}_R(G)$ is that it is the supremal integer $n$ for which there exists an $RG$-module $M$ such that $\homol^n(G;M) \neq 0$.

If $G$ is a group of type $\typeFP (R)$, we have $\mathrm{cd}_R(G) = n$ \iff $\homol^{i}(G;RG) = 0$  for all $i>n$ and $\homol^{n}(G;RG) \neq 0$. This is proved precisely like \cite{Brown1982}*{Proposition VIII.6.7} by changing coefficients from $\Z$ to $R$.

As before, $\mathrm{cd}$ stands for $\mathrm{cd}_\Z$, which is again sensible notation since $\mathrm{cd}_\Z(G) \geqslant \mathrm{cd}_R(G)$ for every group $G$ and every ring $R$.

\subsection{Chain contractions}

Throughout the article, we will make multiple uses of chain contractions, which we introduce here. Let $R$ be a ring and let 
\[
	\cdots \rightarrow C_1 \rightarrow C_0 \rightarrow 0
\]
be a chain complex of free $R$-modules with boundary maps 
\[
	\partial_i \colon C_i \rightarrow C_{i-1}.
\]
A \emph{chain contraction} of $C_\bullet$ is a collection of $R$-module homomorphisms $s_i \colon C_i \rightarrow C_{i+1}$ such that 
\[
	\id_{C_i} = \partial_i s_{i-1} + s_i \partial_{i+1}
\]
for all $i \geqslant 0$, with $s_{-1} = 0$. If there only exist maps $s_i$ satisfying the same condition for $0 \leqslant i \leqslant n$, then we call the collection $s_0, s_1, \dots, s_n$ a \emph{partial chain contraction of length} $n$ of $C_\bullet$. By a slight abuse of terminology, we will also refer to the individual maps $s_i$ as chain contractions. It is clear that if there exists a partial chain contraction of length $n$, then $\homol_i(C_\bullet) = 0$ for all $i \leqslant n$. An easy exercise in homological algebra shows that the converse is also true.

\begin{lem}
	Let $C_\bullet$ be a chain complex of free $R$-modules with no negative terms. Then $\homol_i(C_\bullet) = 0$ for all $i \leqslant n$ if and only if $C_\bullet$ admits a partial chain contraction of length $n$.
\end{lem}

We will often use this result in the following guise.

\begin{cor}
	Let $G$ be a group, let $R$ be a ring, and let $S$ be an $RG$-algebra. Let $C_\bullet \rightarrow R$ be a free resolution of the trivial $RG$-module $R$. Then $\homol_i(G;S) = 0$ for all $i \leqslant n$ if and only if the chain complex $S \otimes_{RG} C_\bullet$ of free left $S$-modules admits a partial chain contraction of length $n$.
\end{cor}

As an immediate consequence we see that if $S'$ is another $RG$-algebra and there is a morphism $\sigma \colon S \to S'$ of $RG$-algebras, then vanishing of $\homol_i(G;S)$ for all $i \leqslant n$ implies the vanishing of $\homol_i(G;S')$ for all $i \leqslant n$. To see this, apply $\sigma$ to a partial chain contraction over $S$.

\subsection{Generalised Sikorav's theorem}

In his thesis \cite{Sikorav1987}*{Chapter IV}, Sikorav proved that if $\phi \colon G \to \Z$ is an epimorphism with finitely generated domain, then $\ker \phi$ is finitely generated \iff 
\[\homol_1(G; \nov \Z G \phi) = 0 = \homol_1(G; \nov \Z G {-\phi}).\]
 Throughout the article, we will abbreviate such vanishing by writing
\[
\homol_1(G; \nov \Z G {\pm \phi})=0,
\]
and similarly for Novikov rings over other coefficients.

Sikorav formulated his theorem in the language of Bieri--Neumann--Strebel (BNS) invariants \cite{Bierietal1987}, which has the advantage of giving meaning to the vanishing of Novikov homology for $\nov \Z G \phi$ only; we are not going to do that, since the introduction of the BNS-invariants is not necessary, and we will always  have the vanishing of Novikov homologies over  $\nov \Z G { \pm \phi}$.

A related result was proved by Ranicki~\cite{Ranicki1995}*{Theorem 1} -- here the statement connects the vanishing of Novikov homologies in all dimensions to higher finiteness properties of the kernel, but only applies in the very restricted situation where $G \cong \ker \phi \times \im(\phi)$.

Combining the initial insight of Sikorav with the work of Bieri, Renz, and Schweitzer \cites{BieriRenz1988,Bieri2007}, the first author proved the following result.

\begin{thm}[Generalised Sikorav's Theorem, \cite{Fisher2024}*{Theorems 5.3 and 6.5}]
	\label{Sikorav}
Let $G$ be a group of type $\typeFP_n(R)$ for some ring $R$, and let $\phi \colon G \to \Z$ be an epimorphism. The following are equivalent:
\begin{enumerate}
\item $\ker \phi$ is of type $\typeFP_n(R)$;
\item $\homol_i(G; \nov  R G {\pm \phi}) = 0$ for all $i \leqslant n$.
\end{enumerate}
\end{thm}
A character $\phi$ satisfying any of the two equivalent conditions will be called \emph{$\typeFP_n(R)$-fibred}.

\begin{rmk}
	\label{fibring open}
	Being $\typeFP_n(R)$-fibred is an open condition, that is, if $\phi$ is $\typeFP_n(R)$-fibred
	then there exists an open neighbourhood $U$ of $\phi$ in $S(G)$ such that every character in $U$ is $\typeFP_n(R)$-fibred. For a proof, see \cite{BieriRenz1988}*{Theorems A and B}.
\end{rmk}

\subsection{Poincar\'e duality and group pairs}
\label{sec PD}

Throughout this subsection, $R$ will be a commutative ring. This assumption seems necessary to ensure that the theory of Poincar\'e duality behaves well under passing to finite-index overgroups. However, we will never perform the operation of passing to such an overgroup, and hence it is possible (indeed, likely) that our discussion holds verbatim for rings $R$ that are not necessarily commutative.

\begin{dfn}[Poincar\'e-duality group]
We will say that a group $G$ is a \emph{Poincar\'e-duality group of dimension $n$ over $R$}, or a \emph{$\PD^n_R$-group} for short, if the group is of type $\typeFP(R)$ and 
\[
 \homol^i(G;R G) \cong \left\{ \begin{array}{cl}
	0 & \textrm{ if } i \neq n, \\
	R & \textrm{ if } i = n,
\end{array} \right.
\]
where the isomorphism is one of $R$-modules. The action of $RG$ on $R = \homol^n(G;R G)$ is the \emph{orientation action}; if it is trivial, the $\PD^n_R$ group is \emph{orientable}.
\end{dfn}

It follows immediately from the definition that $\mathrm{cd}_R(G) = n$. 

The definition above follows that of Bieri~\cites{Bieri1972,Bieri1981} for $R = \Z$, and that of Dicks--Dunwoody~\cite{DicksDunwoody1989book}*{Chapter V} for arbitrary commutative coefficients (though \cite{Bieri1972}*{Section 3.2} briefly discusses the case of a general $R$, quickly focusing on $R = \Q$).

We are not requiring the group $G$ above to be finitely presented -- this way we deviate from the alternative definition of Poincar\'e-duality groups due to Wall \cite{Wall1967}. Also, when $R = \Z$, the underlying ring is often purged from the notation, and one talks simply about (orientable) Poincar\'e-duality groups in dimension $n$.

By a \emph{group pair} $(G, \mathcal H)$ we will understand a pair of a group $G$ and a tuple $\mathcal H$ of subgroups of $G$ (formally, a function from some possibly infinite indexing set to the set of subgroups of $G$). If $\mathcal H$ is non-empty, we define the \emph{associated module} $\Delta_{G / \mathcal H}$ to be the kernel of the augmentation map $\bigoplus_{H \in \mathcal H} R G/H \to R$, with 
\[R G/H = R G \otimes_{R H} R\]
being the free abelian group on the left cosets of $H$, with the obvious $G$ action.

Let $(G, \mathcal H)$ be a group pair and let $M$ be an $RG$-module. The homology and cohomology of the pair $(G,\mathcal H)$ with coefficients in $M$ are defined by 
\begin{align*}
	\homol_i(G,\mathcal H; M) &= \operatorname{Tor}_{i-1}^{RG}(\Delta_{G/\mathcal H},M) \\
	\homol^i(G,\mathcal H; M) &= \operatorname{Ext}_{RG}^{i-1}(\Delta_{G/\mathcal H},M)
\end{align*}
respectively, provided $\mathcal H \neq \varnothing$ (see \cite{BieriEckmann1978}*{Section I.1} and \cite{DicksDunwoody1989book}*{Section V.7}). By convention, $\homol_i(G,\varnothing;M) = \homol_i(G;M)$, and similarly for cohomology. This definition is quite useful, because one immediately sees that given a short exact sequence $0 \rightarrow M_1 \rightarrow M_2 \rightarrow M_3 \rightarrow 0$ of $RG$-modules, there are associated long exact sequences in the homology and cohomology of pairs.

Group pairs appear naturally, as they correspond to the notion of a pair of spaces. Crucially, group pairs satisfy the familiar long exact sequence in cohomology, which is just the long exact sequence in $\operatorname{Ext}^\bullet_{RG}(-,M)$ associated to the short exact sequence
\[
	0 \rightarrow \Delta_{G/\mathcal H} \rightarrow \bigoplus_{H \in \mathcal H} RG/H \rightarrow R \rightarrow 0.
\]

\begin{prop}
	\label{les for pairs}
Let $(G, \mathcal H)$ be a group pair, and let $M$ be an $RG$-module. We have the following long exact sequence in cohomology
\[
\adjustbox{scale=1,center}{
	\begin{tikzcd}[column sep=tiny]
		\cdots \arrow[r]
		& \homol^{n}(G,\mathcal H; M) \arrow[r] \arrow[d, phantom, ""{coordinate, name=ZZ}] 
		& \homol^{n}(G;  M) \arrow[r] 
		& \prod_{H \in \mathcal H} \homol^{n}(H;  M ) \arrow[dll,  rounded corners, to path={ -- ([xshift=2ex]\tikztostart.east)|- (ZZ) [near end]\tikztonodes-| ([xshift=-2ex]\tikztotarget.west)-- (\tikztotarget)}] \\
		& \homol^{n+1}(G,\mathcal H; M) \arrow[r] & \makebox[0pt][l]{$\cdots$} \phantom{\homol^{n}(G;  M)}
	\end{tikzcd}
}
\]
where $M$ is an $RH$-module by restriction.
\end{prop}

When $R = \Z$, the above is \cite{BieriEckmann1978}*{Proposition 1.1}. For general $R$, one easily obtains an analogous long exact sequence in homology.

Given a subgroup $G_0 \leqslant G$, we define $\mathcal H_{G_0}$ to be the tuple obtained from $\mathcal H$ by replacing each $H \in \mathcal H$ by the subgroups of the form $x^{-1} Hx \cap G_0$ where $x$ runs over the double-coset representatives of $H \backslash G/ G_0$. One readily checks that $\Delta_{G/ \mathcal H}$ restricted to being an $R G_0$-module is isomorphic to $\Delta_{G_0 / \mathcal {H}_{G_0}}$.

When $G_0 = \ker \phi$ for a homomorphism $\phi$ with domain $G$, we refer to $\mathcal H_{\ker \phi}$ as $\mathcal H_\phi$.

\begin{dfn}[Poincar\'e-duality pair]
Following \cites{BieriEckmann1978,DicksDunwoody1989book}, given a group $G$ and a non-empty tuple $\mathcal H$ of subgroups of $G$, we say that $(G, \mathcal H)$ is an \emph{orientable Poincar\'e-duality pair of dimension $n$ over $R$} (or \emph{orientable $\PD^n_R$-pair}) if
the group $G$ is of type $\typeFP(R)$ and
\[
\homol^i(G;R G) \cong \left\{ \begin{array}{cl}
	0 & \textrm{ if } i \neq n-1, \\
	\Delta_{G / \mathcal H} & \textrm{ if } i = n-1,
\end{array} \right.
\]
where the isomorphism is that of $RG$-modules.

For convenience, we define $(G,\varnothing)$ to be an  \emph{orientable Poincar\'e-duality pair of dimension $n$ over $R$} if $G$ is an orientable $\PD^n_R$-group.

If $R$ is not specified, then it is taken to be $\Z$.
\end{dfn}

Again, directly from the definition follows that $\mathrm{cd}_R(G) = n-1$. 

Replacing $\Z$ by $R$ in the proof of \cite{BieriEckmann1978}*{Theorem 6.2}, and noting that we are in the orientable case, yields that $(G, \mathcal H)$ for $\mathcal H \neq \varnothing$ is an orientable $\PD_R^n$-pair \iff 
\[
\homol^i(G, \mathcal H;R G) \cong \left\{ \begin{array}{cl}
	0 & \textrm{ if } i \neq n, \\
	R & \textrm{ if } i = n,
\end{array} \right.
\]
as $RG$-modules, where the action on $R$ is trivial, and the pair $(G,\mathcal H)$ is of type $\typeFP(R)$, meaning that $\Delta_{G / \mathcal H}$ admits a finite resolution by finitely generated projective $RG$-modules (see also \cite{DicksDunwoody1989book}*{Definition V.7.1}).

\begin{prop}
	\label{peripheral PD}
If $(G,\mathcal H)$ is an orientable $\PD_R^{n}$-pair with $\mathcal H \neq \varnothing$, then $\mathcal H$ is finite, and every $H \in \mathcal H$ is itself an orientable  $\PD_R^{n-1}$-group.
\end{prop}
\begin{proof}
	When $R = \Z$, this is \cite{BieriEckmann1978}*{Theorem 4.2(ii) and (iii)}. For general commutative $R$, the second assertion is covered by \cite{DicksDunwoody1989book}*{Theorem V.7.11}.
	
	Since the $RG$-module $\Delta_{G / \mathcal H}$ admits a finite resolution by finitely generated projective $RG$-modules, it must be finitely generated, which is only possible if $\mathcal H$ is finite.
\end{proof}

\begin{rmk}
	\label{PD}
Somewhat unsurprisingly, Poincar\'e-duality groups and pairs have the familiar duality between between homology and cohomology. Concretely, given an $RG$-module $M$, a possibly empty $\mathcal H$, and an orientable $\PD_R^n$-pair $(G,\mathcal H)$ we have natural isomorphisms
\[
\homol^i(G, \mathcal H; M) \cong \homol_{n-i}(G; M),
\]
for all integers $i$, and similarly starting from homology of the pair, see \cite{BieriEckmann1978}*{Definition 4.1 and Theorem 6.2} and \cite{DicksDunwoody1989book}*{Definition V.7.1}.
\end{rmk}

As in the case of finiteness properties, if $(G,\mathcal H)$ is an orientable $\PD_R^n$-pair and $S$ is an $R$-algebra, then $(G,\mathcal H)$ is an orientable  $\PD_S^n$-pair (see \cite{DicksDunwoody1989book}*{Proposition V.3.7}). We will use this fact in a situation where  $R$ is a prime field and $S$ is another field of the same characteristic.

\subsection{RFRS}

A group $G$ is \emph{residually finite rationally solvable (RFRS)} if it admits a normal residual chain $G = G_0 \geqslant G_1 \geqslant G_2 \geqslant \dots$ of finite-index subgroups such that  
\[
	\ker(G_i \rightarrow \Q \otimes_\Z G_i/[G_i,G_i]) \leqslant G_{i+1}
\]
for all $i \geqslant 0$. For finitely generated groups, a simpler characterisation was recently obtained by Okun, Schreve, and the first and third authors: a finitely generated group $G$ is RFRS if and only if it is residually \{virtually abelian and locally indicable\} \cite{OkunSchreve2025}*{Theorem 6.3}. Since for finitely generated groups being virtually abelian and locally indicable is the same as being virtually abelian and poly-$\Z$, we see that  a finitely generated group $G$ is RFRS if and only if it is residually \{virtually abelian and poly-$\Z$\} as well.

The RFRS condition was introduced by Agol in \cite{Agol2008}, where he proved that a compact, irreducible $3$-manifold $M$ with virtually RFRS fundamental group virtually fibres over $S^1$ if and only if $\chi(M) = 0$.

\subsection{\texorpdfstring{$L^2$}{L\texttwosuperior}-Betti numbers over fields}

Since our focus in this paper is on RFRS groups, we will present the algebraic approach to $L^2$-homology via Linnell skew-fields. We will also need the positive characteristic version of $L^2$-homology, that is (for now) only defined in an algebraic way.

When $G$ is RFRS and $\K$ is a field, Jaikin-Zapirain in \cite{Jaikin-Zapirain2021} constructs a skew-field $\D_{\K G}$ that contains $\K G$ as a subring, and has two additional properties: the embedding $\K G \subseteq \D_{\K G}$ is \emph{Hughes-free}, and  the skew-field $\D_{\K G}$ is \emph{universal}. The precise meaning of these adjectives is not really essential for us here; one important consequence of Hughes-freeness is that $\D_{\K G}$ is unique up to isomorphisms of skew-fields containing $\K G$ that fix $\K G$ (this result is due to Hughes~\cite{Hughes1970}). We will therefore treat $\D_{\K G}$ as well-defined unique objects.

When $\K = \Q$, the construction of $\D_{\Q G}$ is older, as it goes back to Linnell~\cite{Linnell1993}, and it works in a much greater generality than that of RFRS groups. Jaikin-Zapirain's definition also covers a larger class of groups.

\begin{dfn}
	Given a RFRS group $G$ and a field $\K$, we define
	\[
	\betti{G}{i}{\K} = \dim_{\D_{\K G}} \homol_i(G; \D_{\K G}) \in \N \sqcup \{ \infty\}
	\]
	to be the \emph{$i$th $L^2$-Betti number of $G$ over $\K$}.
\end{dfn}

The notation stems from the fact that $\betti{G}{i}{\Q}$ coincides with the classical $i$th $L^2$-Betti number $\bettiQ{G}{i}$. 
More information about this can be found in \cite{Linnell1993} and \cite{Lueck2002}*{Chapter 10}.

 One of the key properties of $L^2$-Betti numbers over fields is the way they behave under passing to finite index.
 
 \begin{lem}[\cite{Fisher2024}*{Lemma 6.3}]
 	\label{betti nrs scaling}
 Let $G_0$ be a finite-index subgroup of a RFRS group $G$. For all $i$ and all fields $\K$ we have
\[
\betti{G_0}{i}{\K} = |G:G_0| \betti{G}{i}{\K}.
\]
 \end{lem}

Another important property of $L^2$-Betti numbers is that they vanish for mapping tori.

\begin{thm}[Mapping Torus Theorem, \cite{Fisher2024}*{Theorem 6.4}]
\label{mapping torus}
Suppose that $G$ is a RFRS group, that $\phi \colon G \to \Z$ is an epimorphism, and let $\K$ be a field. If $\ker \phi$ is of type $\typeFP_n(\K)$, then $\betti{G}{i}{\K} = 0$ for every $i \leqslant n$.
\end{thm}

For $\K = \Q$ this is a celebrated theorem of L\"uck~\cite{Lueck1994a} (see also \cite{Lueck2002}*{Theorem 7.2}), which holds for many more groups.

The stated theorem can be generalised also in another way: 
instead of requiring $\ker \phi$ to be of type $\typeFP(\K)$, if we insist on $\betti{\ker \phi}{i}{\K}<\infty$ for all $i$, then the conclusion holds even when  the quotient $\Z$ is replaced by any infinite amenable group \cite{FisherKlinge2024}*{Proposition 3.18} -- when $\K = \Q$, this result is due to Gaboriau \cite{Gaboriau2002}*{Th\'eor\`eme 6.6}.

\smallskip

The Mapping Torus Theorem has a partial inverse, the Virtual Algebraic  Fibring Theorem. We will need more than the final statement, and to this end we need to introduce a couple of additional objects.

There are a number of consequences of Hughes-freeness: when $H$ is a subgroup of $G$, we have a natural embedding $\D_{\K H} \subseteq \D_{\K G}$. If $H$ is normal, 
the conjugation action of $G$ on $H$ induces an action of $G$ on $\D_{\K H}$, and this can be used to define a crossed product $\D_{\K H} G/H$. The crossed product naturally sits inside $\D_{\K G}$, and when $H$ is of finite index, the two rings coincide \cite{Jaikin-Zapirain2021}*{Proposition 2.2}. 

Let $G_0$ be a finite-index normal subgroup of $G$, and let $\phi \colon G_0 \to \R$ be a non-trivial homomorphism. We may view $\K G_0$ as a crossed product $(\K \ker \phi) \im(\phi)$. This ring sits inside $\D_{\K \ker \phi} (\im \phi)$, and its Ore localisation, which is $\D_{\K G_0}$, lies inside the Malcev--Neumann completion of $\D_{\K \ker \phi} (\im \phi)$ with respect to the biorder on $\im \phi$ inherited from $\R$. This last ring contains $\nov {\D_{\K \ker \phi}} {\im(\phi)} \id$, which in turn contains $\nov \K {G_0} \phi$. All of this yoga shows that it makes sense to consider the intersection of $\D_{\K G_0}$ and $\nov \K {G_0} \phi$. For an element of $\D_{\K G_0}$, we will say that it \emph{lies in $\nov \K {G_0} \phi$} if it is contained in this intersection.

Given a subset $U \subseteq S(G_0)$, we will say that an element of $\D_{\K G_0}$ \emph{lies in $\nov \K {G_0} U$} if it lies in $\nov \K {G_0} \phi$ for all $\phi \in U$. All such elements form a  ring that we will denote by $\D_{\K G_0} \cap \nov \K {G_0} U$. The key point here is that for different characters $\phi$, we are intersecting $\D_{\K G_0}$ and $\nov \K {G_0} \phi$ inside different ambient rings.

When $U$ is invariant under the conjugation action of $G$ (or, equivalently, of $G/G_0$), then we may form a crossed product 
\[
	(\D_{\K G_0} \cap \nov \K {G_0} U) G/G_0.
\]
Since $\D_{\K G} = \D_{\K G_0} G/G_0$, we may again talk about elements of $\D_{\K G}$ \emph{lying in} $(\D_{\K G_0} \cap \nov \K {G_0} U) G/G_0$.

\begin{thm}\label{virt alg fibring}
	Let $G$ be a finitely generated RFRS group, and let $\K$ be a field. For every finite subset $S \subseteq \D_{\K G}$ there exists a normal subgroup $G_0 \leqslant G$ of finite index and an open subset $U \subseteq S(G_0)$ such that
	\begin{enumerate}
		\item $U$ is invariant under the conjugation action of $G$ and under the antipodal map;
		\item $\mathring{\overline U}$, the interior of the closure of $U$, contains $S(G)$;
		\item all the elements of $S$ lie in $(\D_{\K G_0} \cap \nov \K {G_0} U) G/G_0$.
	\end{enumerate}
\end{thm}

The theorem is not stated this way anywhere. It is a slight weakening of \cite{Kielak2020a}*{Theorem 4.13} and \cite{Jaikin-Zapirain2021}*{Theorem 5.10}, see also \cites{Fisher2024,OkunSchreve2025}. 

The Virtual Algebraic Fibring Theorem is a direct consequence.

\begin{thm}[Virtual Algebraic Fibring Theorem, \cite{Fisher2024}*{Theorem 6.6}]
	\label{fibring thm}
Let $\K$ be a field and let $G$ be a virtually RFRS group of type $\typeFP_{n}(\K)$. The following are equivalent.
\begin{enumerate}
	\item There exists a finite-index subgroup $G_0$ in $G$ and an epimorphism $\phi \colon G_0 \to \Z$ such that $\ker \phi$ is of type $\typeFP_{n}(\K)$;
	\item $\betti{G}{i}{\K} = 0$ for all $i \leqslant n$.
\end{enumerate} 
\end{thm}

\section{Main tools}

\subsection{Novikov cohomology}

The following observation appeared already in \cite{Bridsonetal2025}.

\begin{prop}\label{prop:nov_homol_cohomol}
	Let $C_\bullet$ be a chain complex of free left $R G$-modules without negative terms. If $\homol_i (\nov R G \phi \otimes_{R G} C_\bullet) = 0$ for all $i \leqslant n$ then 
	\[
	\homol^i (\Hom_{RG}( C_\bullet, \nov R G {\phi})) = 0
	\]
	 for all $i \leqslant n$.
\end{prop}
\begin{proof}
	Let $C_\bullet = (C_i, \partial_i)$. The vanishing of $\homol_i (\nov R G \phi \otimes_{R G} C_\bullet)$ allows us to construct a partial chain contraction, that is $\nov R G \phi$-homomorphisms $s_i \colon \nov R G \phi \otimes_{R G} C_i \to \nov R G \phi \otimes_{R G} C_{i+1}$ for $i \leqslant n$ such that 
	\[
	\partial_i s_{i-1} + s_i \partial_{i+1} = \id_{\nov R G \phi \otimes_{R G} C_i}
	\]
	for $i \leqslant n$, with $s_{-1}=0$. (We are abusing notation slightly, since we use $\partial_i$ here to mean its natural extension  to $\nov R G \phi \otimes_{R G} C_i \to\nov R G \phi \otimes_{R G} C_{i-1} $.)
	
	Take $i \leqslant n$ and an $i$-cocycle $z \colon C_i \to \nov R G \phi$. Since $\nov R G \phi$ is a left $\nov R G \phi$-module, the cocycle $z$ extends to
	\[
	\overline z  \colon  \nov R G \phi \otimes_{R G} C_i \to  \nov R G \phi,  \ \xi \otimes c \mapsto \xi.z(c) .\]
	We get
	\[
	\overline z =  (\partial_i s_{i-1} + s_i \partial_{i+1}) \overline z =   \partial_i s_{i-1} \overline z,
	\]
	and restricting both sides to $C_i$ gives
	\[
	z = \partial_i  \circ (s_{i-1}  \overline z)\vert_{C_i},
	\] 
	showing that $z$ is a coboundary.
\end{proof}

Applying the above to $C_\bullet$ being a free resolution of $R$ yields the following.

\begin{cor}
	\label{flip}
	Let $G$ be a group and $\phi \colon G \to \R$ a homomorphism. If $\homol_i(G; \nov R G \phi)=0$ for all $i \leqslant n$, then $\homol^i(G; \nov R G {-\phi})=0$ for all $i \leqslant n$.
\end{cor}

The change of coefficients from $\nov R G {\phi}$ to $\nov R G {-\phi}$ is an artifact of notation: as explained in \cref{sec modules}, the module $\nov R G {\phi}$ in $\homol_i(G; \nov R G {\phi})$ is a left  $R G$-module. But when computing homology, we tensor $C_\bullet$ on the left by the right module $(\nov R G {\phi})^*$, isomorphic to $\nov R G {-\phi}$ as a right $R G$-module.

Combining the above with \cref{Sikorav} yields the following.

\begin{cor}
	\label{Sikorav hom and cohom}
	Let $G$ be a group and $\phi \colon G \to \Z$ an epimorphism. The following are equivalent:
	\begin{enumerate}
		\item $\ker \phi$ is of type $\typeFP_n(R)$;
		\item $\homol^i(G; \nov R G {\pm \phi})=0 = \homol_i(G; \nov R G {\pm \phi})$ 	 for all $i \leqslant n$.
	\end{enumerate}
\end{cor}

\subsection{Long exact sequence}

We will now introduce the key tool.

\begin{prop}
	\label{les}
	Let $R$ be a ring. Let $G$ be a group with an epimorphism $\phi \colon G \to \Z$ with kernel $K$. Suppose that $\mathcal H$ is a (possibly empty) tuple of subgroups of $G$. The following long sequence of $R K$-modules is exact.
	\[
	\adjustbox{scale=1,center}{
		\begin{tikzcd}[column sep=tiny]
			\cdots \arrow[r]
			 & \homol^{n}(G,\mathcal H;  R G) \arrow[r] \arrow[d, phantom, ""{coordinate, name=ZZ}] 
			 & \homol^{n}(G,\mathcal H;  \nov R G {  \phi} \oplus \nov R G {  -\phi}) \arrow[r] 
			 & \homol^{n}(K,\mathcal H_\phi;  R K ) \arrow[dll,  rounded corners, to path={ -- ([xshift=2ex]\tikztostart.east)|- (ZZ) [near end]\tikztonodes-| ([xshift=-2ex]\tikztotarget.west)-- (\tikztotarget)}] \\
			& \homol^{n+1}(G,\mathcal H;  R G) \arrow[r] & \makebox[0pt][l]{$\cdots$} \phantom{\homol^{n}(G,\mathcal H;  \nov R G {  \phi} \oplus \nov R G {  -\phi})}
		\end{tikzcd}
	}
	\]
\end{prop}
\begin{proof}
	Using the notation of \cref{section Novikov}, consider the short exact sequence of $RG$-bimodules
	\[
	0 \to RG \to \nov R G \phi \oplus \nov R G {-\phi} \to \prod_{i \in \Z} R K t^i \to 0,
	\] 
	where the second map is the diagonal embedding, and the third map is the difference of the embeddings. If $\mathcal H$ is empty, then an application of Shapiro's lemma shows that the desired long exact sequence is just the long exact sequence in cohomology associated to the short exact sequence above. If $\mathcal H$ is non-empty, then the long exact sequence in the statement of the theorem is just the long exact sequence in $\operatorname{Ext}_{RG}^\bullet(\Delta_{G/\mathcal H},-)$ associated to the short exact sequence above. This is because
	\[
		\operatorname{Ext}_{RG}^\bullet\left(\Delta_{G/\mathcal H},\prod_{i \in \Z} R K t^i\right) \ \cong \ \operatorname{Ext}_{RK}^\bullet(\Delta_{G/\mathcal H}, R K)
	\]
	by Shapiro's lemma, and since the restricted $RK$-module $\Delta_{G/\mathcal H}$ is just $\Delta_{K/\mathcal H_\phi}$, we have
	\[
		\operatorname{Ext}_{RK}^{n}(\Delta_{G/\mathcal H}, RK) \ \cong \ \homol^{n+1}(K,\mathcal H_\phi; RK),
	\]
	as desired. \qedhere
\end{proof}

\begin{rmk}\label{rmk: LES_cd}
	There is a more general long exact sequence that is useful to study the cohomological dimension of co-abelian subgroups, which was used in \cite{Fisher2024_cdKernels}. Let $(G,\mathcal H)$ be a group pair, let $\phi \colon G \rightarrow \R$ be a character with kernel $K$. Moreover, let $M$ be any left $RK$-module. The coinduced module $\prod_{t \in G/K} Mt$ has a pair of submodules $L^{\pm\phi}$ consisting of the sequences $(m_t t)$ for which there exists $\alpha \in \R$ such that $m_t = 0$ for all $\pm t < \pm\alpha$ (where the order is taken in $\R \supseteq G/K$). The inclusions $L^{\pm\phi} \hookrightarrow \prod_{t \in G/K} Mt$ induce a short exact sequence
	\[
		0 \rightarrow N \rightarrow L^\phi \oplus L^{-\phi} \rightarrow \prod_{t \in G/K} Mt \rightarrow 0
	\]
	with $N$ being the $RG$-module induced from $M$, 
	whence we obtain long exact sequences in $\homol^n(G;-)$ and $\homol^n(G,\mathcal H;-)$.
\end{rmk}

\subsection{Reduction to finitely many primes}

Suppose that $\mathcal P$ is a collection of primes. We denote by $\Z_\mathcal P$ the ring obtained by localising $\Z$ at the multiplicative set
\[
\left\{ \prod_{i=1}^m {p_i}^{n_i} \mid m, n_i, \in \N, p_i \in \mathcal P \right\}.
\]	 
A more pedestrian description of $\Z_\mathcal P$ is as the subring of $\Q$ whose elements can be expressed as fractions with denominators being products of powers of primes in $\mathcal P$. 

Given a matrix $M$ over $\novq{G}{\phi}$, we say that $M$ \emph{involves finitely many primes} if and only if there is a finite set of primes $\mathcal P$ such that the entries of $M$ are all in $\nov{\Z_{\mathcal P}}{G}{\phi}$.

\begin{prop}\label{prop:large_primes}
	Let $G$ be a group, let $\varphi \colon G \rightarrow \R$ be a non-trivial character, and let $C_\bullet$ be a chain complex  of finitely generated projective left $\Z G$-modules without negative terms. If $\homol_i(\novq{G}{\phi} \otimes_{\Z G} C_\bullet) = 0$ for all $i \leqslant n$, then there is a finite collection of primes $\mathcal P$ such that 
	\[
	\homol_i(\nov{\Z_\mathcal P}{G}{\phi} \otimes_{\Z G} C_\bullet) = 0
	\]
	for all $i \leqslant n$.
\end{prop}
\begin{proof}
	By taking direct sums with projective modules, we can modify the chain complex (without changing its homology over arbitrary coefficients) so that it consists of finitely generated free modules. We fix bases for the free modules and identify the boundary maps with matrices over the appropriate rings. We will inductively define chain contractions 
	\[
	s_i \colon \novq{G}{\phi} \otimes_{\Z G} C_i \rightarrow \novq{G}{\phi} \otimes_{\Z G} C_{i+1}
	\]
	for $i = 0, \dots, n$ involving finitely many primes.

	We set $s_{-1}$ to be the zero map, which trivially involves finitely many primes. For some $0 \leqslant i \leqslant n$, suppose we have defined chain contractions $s_{0}, \dots, s_{i-1}$ involving finitely many primes. By acyclicity of $\novq{G}{\phi} \otimes_{\Z G} C_\bullet$, there is a chain contraction 
	\[
	\sigma_i \colon \novq{G}{\phi} \otimes_{\Z G} C_i \rightarrow \novq{G}{\phi} \otimes_{\Z G} C_{i+1},
	\]
	such that $\partial_i s_{i-1} + \sigma_i \partial_{i+1} = \id$. We will modify $\sigma_i$ to obtain a new chain contraction involving only finitely many primes.  We form $\overline{\sigma_i}$ by truncating the entries of $\sigma_i$ at a  real number sufficiently large so that 
	\[
	A =  \id -\partial_i s_{i-1} - \overline{\sigma_i} \partial_{i+1} 
	\]
	has all entries of positive support. We have $A = (\sigma_i - \overline{\sigma_i}) \partial_{i+1}$, and $\sigma_i - \overline{\sigma_i}$ is a matrix over $\Q G$, so the matrix $A$ involves finitely many primes. Since the supports of the entries of $A$ are positive, $\id - A$ is invertible with inverse $\sum_{i=0}^\infty A^i$. Note that 
	\[
	A \partial_i = (\sigma_i-\overline{\sigma_i}) \partial_{i+1} \partial_i = 0,
	\]
	so 
	\[
	\id = (\id-A)^{-1} (\partial_i s_{i-1} +  \overline{\sigma_i} \partial_{i+1} )= \partial_i s_{i-1} + (\id-A)^{-1} \overline{\sigma_i} \partial_{i+1}.
	\]
	Thus, we put $s_i = (\id - A)^{-1} \overline{\sigma_i}$, and observe that $s_i$ involves only finitely many primes. Letting $\mathcal P$ be the collection of primes involved in the chain contractions $s_{0}, \dots, s_n$, we conclude that 
	\[
		\homol_i(\nov{\Z_\mathcal P}{G}{\phi} \otimes_{\Z G} C_\bullet) = 0
	\]
	for $i \leqslant n$. \qedhere
\end{proof}

\begin{cor}\label{cor:large_primes}
	Let $G$ be a group of type $\typeFP_n$ and let $\phi \colon G \rightarrow \R$ be a non-trivial homomorphism.
	\begin{enumerate}
		\item\label{item:large_primes_homol} If $\homol_i(G;\novq{G}{\phi}) = 0$ for $i \leqslant n$, then there is a finite collection of primes $\mathcal P$ such that $\homol_i(G;\nov{\Z_\mathcal P}{G}{\phi}) = 0$.
		\item\label{item:large_primes_cohomol} If $\operatorname{cd}(G) = n$ and there is an integer $0 \leqslant j \leqslant n$ such that 
		\[
		\homol^i(G;\novq{G}{\phi}) = 0
		\]
		for $j \leqslant i \leqslant n$, then there is a finite collection of primes $\mathcal P$ such that $\homol^i(G;\nov{\Z_\mathcal P}{G}{\phi}) = 0$ for $j \leqslant i \leqslant n$.
	\end{enumerate}
\end{cor}
\begin{proof}
	Item (\ref{item:large_primes_homol}) is an immediate consequence of \cref{prop:large_primes}. In the situation of item (\ref{item:large_primes_cohomol}), there is a resolution $0 \rightarrow P_n \rightarrow \dots \rightarrow P_0 \rightarrow \Z \rightarrow 0$ by finitely generated projective $\Z G$-modules. The Novikov cohomology of $G$ is
	\[
	\homol^i(G;\nov{R}{G}{\phi}) = \homol^i(\Hom_{\Z G}(P_\bullet,\Z G) \otimes_{\Z G} \nov{R}{G}{\phi})
	\]
	for any ring $R$, because the modules $P_i$ are finitely generated. The result then follows by applying \cref{prop:large_primes} to the chain complex
	\[
	\cdots \rightarrow \Hom_{\Z G}(P_{n-1},\Z G) \rightarrow \Hom_{\Z G}(P_n,\Z G) \rightarrow 0
	\]
	of finitely generated projective $\Z G$-modules. \qedhere
\end{proof}

\section{Acyclicity at infinity}

We will now apply the tools from the previous section to study homology at infinity of algebraic fibres.

\begin{thm}\label{thm:cohomol_fibre}
	Let $R$ be a ring.
	Let $G$ be a group and $\phi \colon G\to \Z$ an epimorphism. If $K = \ker \phi$ is of type $\typeFP_n(R)$ then 
	\[
	\homol^{i+1}(G;RG) \cong \homol^i(K; R K)
	\]
	as $R K$-modules for all $i < n$ and $\homol^n(K; R K)$ is isomorphic as an $R K$-module to a submodule of $\homol^{n+1}(G;RG)$.
\end{thm}
\begin{proof}
	This follows immediately from \cref{Sikorav hom and cohom,les}, taking $\mathcal H = \varnothing$ and  noting that 
	\[\homol^{i}(G; \nov R G {  \phi} \oplus \nov R G {  -\phi}) = \homol^{i}(G; \nov R G {  \phi}) \oplus \homol^{i}(G;\nov R G {  -\phi}). \qedhere\]
\end{proof}

\begin{rmk}
	While we find the approach using Novikov rings more straightforward, there is also a spectral sequence argument, which we sketch here. The Lyndon--Hochschild--Serre spectral sequence associated to the extension $\mathbbm 1 \rightarrow K \rightarrow G \rightarrow \Z \rightarrow \mathbbm 1$ is
	\[
		E_2^{p,q} = \homol^p(\Z; \homol^q(K; RG)) \Rightarrow \homol^{p+q}(G;RG).
	\]
	Because $K$ is of type $\typeFP_n(R)$, we have
	\[
		\homol^q(K;RG) \cong \homol^q(K; \bigoplus_{i\in\Z} RKt^i) \cong \bigoplus_{i\in \Z} \homol^q(K; RK) t^i
	\]
	for all $q \leqslant n$.
	Thus,
	\[
		\homol^0(\Z; \homol^q(K; RG)) = 0 \quad \text{and} \quad \homol^1(\Z; \homol^q(K;RG)) \cong \homol^q(K;RK)
	\]
	for $0 \leqslant q \leqslant n$ by an easy cohomology computation. The page $E_2^{p,q}$ is concentrated in the columns $0$ and $1$, so it follows immediately that $\homol^q(K;RK) \cong \homol^{q+1}(G;RG)$ for $q < n$ and that \[\homol^n(K;RK) \leqslant \homol^{n+1}(G;RG).\]
\end{rmk}

\begin{cor}\label{cor:H2zero_oneEndFibre}
	Let $\K$ be a field.
	Let $G$ be a group with $\homol^{2}(G;\K G) = 0$. If $\phi \colon G \to \Z$ is an epimorphism with finitely generated kernel, then $\ker \phi$ has at most one end.
\end{cor}
\begin{proof}
	Since $\ker \phi$ is finitely generated, it is of type $\typeFP_1(\K)$. Hence $\homol^1(\ker \phi; \K \ker \phi) = 0$.
	By \cite{Swan1969}*{Lemma 3.5}, this statement being true for one field implies  it being true for all fields, and in particular for the field of two elements $\F_2$. Finally, $\homol^{1}(\ker \phi;\F_2 \ker \phi) = 0$ implies that $\ker \phi$ has at most one end -- this is  due to Specker~\cite{Specker1949}. 
\end{proof}

\begin{maincor}[\ref{one-ended fibr}]
	Let $M$ be a closed aspherical manifold of dimension at least three. If $\phi \colon \pi_1(M) \to \Z$ is an epimorphism with finitely generated kernel, then the kernel has only one end. 
\end{maincor}
\begin{proof}
	The fact that $\homol^{2}(\pi_1(M);\Q \pi_1(M)) = 0$ follows directly from Poincar\'e duality and the fact that   $\homol_i(\pi_1(M);\Q \pi_1(M)) = 0$ for all $i>0$, which follows from Shapiro's lemma.
	
	We conclude that $\ker \phi$ has at most one end. If it had zero ends, it would be finite, but $\pi_1 M$ is torsion free, and hence $\ker \phi$ would be trivial, forcing $\pi_1 M \cong \Z$. But this implies that the cohomological dimension of $\pi_1 M$ is one, contradicting the existence of the fundamental class of $M$.
\end{proof}

This confirms the intuition of Douba, expressed after the statement of \cite{Douba2024}*{Theorem 1}, that his Zariski-dense  subgroup of $\mathrm{SL}_5(\Z)$ that is finitely generated but not finitely presented has one end -- the subgroup arises precisely as a kernel of an epimorphism $\phi \colon \pi_1(M) \to \Z$ where $M$ is a closed four-dimensional hyperbolic manifold.

\begin{rmk}
	 A natural source of examples of groups with \[\homol^2(G;\K G) = 0,\]
	  with $\K$ being a  field, is provided by the class of duality groups. A group $G$ of type $\typeFP$ is an $n$-dimensional \emph{duality group} if $\homol^i(G;\Z G) = 0$ for all $i \neq n$ and $\homol^n(G;\Z G)$ is torsion-free as a $\Z$-module. Examples of duality groups include Poincar\'e-duality groups, and therefore fundamental group of closed aspherical manifolds, but they also include the torsion-free finite-index subgroups of arithmetic groups \cite{BorelSerre1973}, mapping class groups \cite{Harer1986}, and outer automorphism groups of free groups \cite{BestvinaFeighn2000}. Note also that if $(G,\mathcal H)$ is a Poincar\'e-duality pair of dimension at least $4$ over some field $\K$, then $\homol^2(G;\K G)$ vanishes. Thus, algebraic fibres in any of the groups listed above are one-ended.
\end{rmk}

\section{Virtual algebraic fibring}

\subsection{RFRS groups}

We now give a sharpening of \cref{virt alg fibring} which holds for group pairs.

\begin{prop}\label{prop:tuple_sigma}
	Let $G$ be a RFRS group of type $\typeFP_n(\K)$ for some field $\K$ and let $\mathcal H$ be a (possibly empty) finite tuple of subgroups of $G$, where each $H \in \mathcal H$ is of type $\typeFP_n(\K)$. If
	\[
		\betti{G}{i}{\K} = \betti{H}{i}{\K} = 0
	\]
	for all $i \leqslant n$ and each $H \in \mathcal H$, then there is a finite-index normal subgroup $G_0 \leqslant G$ and an antipodally symmetric open set $U \subseteq S(G_0)$, invariant under the conjugation action of $G$, such that $\mathring{\overline U} \supseteq S(G)$ and 
	\[
		\homol_i (G_0; \nov \K {G_0} {\phi}) = \homol_i (H; \nov \K {H} {\phi|_H}) = 0
	\]
	for all $i \leqslant n$, $H \in \mathcal H_{G_0}$, and $\phi \in U$.
\end{prop}
\begin{proof}
	We pick a free resolution of the trivial $\K G$-module $\K$ in which the modules up to degree $n$ are finitely generated, and endow the free modules with bases; we pick analogous resolutions for every group $H \in \mathcal H$. Since the homologies of the groups $G$ and $H$ over $\D_{\K G}$ and $\D_{\K H}$, respectively, vanish up to degree $n$, we may build partial chain contractions witnessing this vanishing. Each such partial chain contraction consists of $n+1$ finite matrices over $\D_{\K G}$, since $\D_{\K H}$ is naturally a sub-skew-field of $\D_{\K G}$; since $\mathcal H$ is finite, altogether this gives finitely many finite matrices over $\D_{\K G}$. By \cref{virt alg fibring}, there exists a finite-index normal subgroup $G_0$ of $G$ and an antipodally symmetric open set $U \subseteq S(G_0)$ such that all of the entries of the matrices lie over $(\D_{\K G_0} \cap \nov \K {G_0} U) G/G_0$. Moreover, $\mathring{\overline U} \supseteq S(G)$ and $U$ is invariant under the conjugation action of $G$ on $S(G_0)$.
		
	We take a map $\phi \colon G_0 \to \R$ from $U$, and obtain $\homol_i (G_0; \nov \K {G_0} { \phi}) = 0$ for every $i\leqslant n$, and  
	\[
		\homol_i (G_0\cap g^{-1}Hg; \nov \K {G_0} { \phi}) = 0
	\]
	for every $H \in \mathcal H$, $g \in G$, and $i \leqslant n$ (the fact that we may conjugate $H$ comes from conjugacy-invariance of $U$). Hence, $\homol_i (H; \nov \K {G_0} {\phi}) = 0$ for every $i \leqslant n$ and every $H \in \mathcal H_{G_0}$. By \cref{lem:restrict_novikov}, $\homol_i (H; \nov \K H {\phi\vert_H}) = 0$ for every $i \leqslant n$ and every $H \in \mathcal H_{G_0}$. \qedhere
\end{proof}

Since we are interested in both virtual algebraic fibring over a specific field, and over all fields simultaneously, for notational convenience let us consider a set $\mathcal K$ of fields with a specific property.

\begin{dfn}
	We say that a set of fields $\mathcal K$ is \emph{small} if $\mathcal K$ is finite or if it contains $\Q$ and for every prime $p$ it contains only finitely many fields of characteristic $p$.
	
	If $\mathcal K$ is finite, we will say that a group is of \emph{type $\typeFP_{n}(\mathcal K)$} if it is of type  $\typeFP_{n}(\K)$ for every $\K \in \mathcal K$; similarly, we define $\PD_{\mathcal K}^n$ to mean $\PD_{\K}^n$ for every $\K \in \mathcal K$.
		
	When $\mathcal K$ is infinite, we  will say that a group is of \emph{type $\typeFP_{n}(\mathcal K)$} if it is of type  $\typeFP_{n}$; similarly, we define $\PD_{\mathcal K}^n$ to mean $\PD^n$.
\end{dfn}

\begin{thm}
	\label{l2 all chars}
	Let $\mathcal K$ be a small set of fields.
	Let $G$ be a RFRS group of type $\typeFP_{n}(\mathcal K)$, and let $\mathcal H$ be a possibly empty finite tuple of subgroups of $G$, all of type $\typeFP_{m}(\mathcal K)$. The following are equivalent:
	\begin{enumerate}
		\item \label{item virt fib} There exists a finite-index subgroup $G_0$ of $G$ and an epimorphism $\phi \colon G_0 \to \Z$ that does not vanish on any of the subgroups in $\mathcal H_{G_0}$, such that for all $\K \in \mathcal K$ the group $\ker \phi$ is of type $\typeFP_n(\K)$, and $\ker \phi|_{H'}$ is of type $\typeFP_m(\K)$ for every $H' \in \mathcal H_{G_0}$;
		\item \label{item l2 vanishing} We have $\betti{G}{i}{\K} = 0$ for all $i\leqslant n$ and all $\K \in \mathcal K$, and $\betti{H}{i}{\K} = 0$ for all $i\leqslant m$, all $\K \in \mathcal K$, and all $H \in \mathcal H$.
	\end{enumerate}
\end{thm}
\begin{proof}
	We will prove both implications in turn.
	\begin{enumerate}
		\item[\eqref{item virt fib} $\Rightarrow$ \eqref{item l2 vanishing}]
		This follows immediately from \cref{mapping torus,betti nrs scaling}.
		\item[\eqref{item l2 vanishing} $\Rightarrow$ \eqref{item virt fib}]
		We first suppose that $\mathcal K$ is infinite, and so $\Q \in \mathcal K$. By \cref{Sikorav} and \cref{prop:tuple_sigma}, there exists a finite-index subgroup $G_0 \leqslant G$ and an epimorphism $\phi \colon G_0 \to \Z$ such that $\ker \phi$ is of type $\typeFP_n(\Q)$ and $\ker \phi\vert_{H}$  is of type $\typeFP_m(\Q)$ for every $H \in \mathcal H_{G_0}$.
		
		Applying  \cref{prop:large_primes} to $G_0$ and each group in $\mathcal H_{G_0}$ in turn for $\phi$ and $-\phi$, and then using \cref{Sikorav}, there exists a finite collection of primes $\mathcal P$ such that all the kernels above are of type $\typeFP_n(\Z_{\mathcal P})$ or $\typeFP_m(\Z_{\mathcal P})$, respectively, and hence of type $\typeFP_n(\F_p)$ or $\typeFP_m(\F_p)$ for all primes not in $\mathcal P$, since  there is a ring epimorphism $\Z_\mathcal P \rightarrow \F_p$ for all primes $p \notin \mathcal P$. (It is precisely \cref{prop:large_primes} that forces us to assume that $G$ is of type $\typeFP_n$ and every $H \in \mathcal H$ is of type $\typeFP_m$.)
		
		Being of type $\typeFP_n(\F_p)$ or $\typeFP_m(\F_p)$, the kernels are automatically of type $\typeFP_n(\K)$ or $\typeFP_m(\K)$, respectively, for every field $\K$ of characteristic $p$. Similarly, being of type $\typeFP_n(\Q)$ or $\typeFP_m(\Q)$ takes care of all fields of characteristic $0$.
		
		Let $\mathcal L$ be the subset of $\mathcal K$ consisting of the fields of characteristics lying in $\mathcal P$. Since $\mathcal K$ is small, the set $\mathcal L$ is finite. Our goal now is to show that, up to passing to a further finite-index subgroup, we may arrange for our kernels to also be of type $\typeFP_n(\K)$ or $\typeFP_m(\K)$ for every field $\K \in \mathcal L$. If $\mathcal K$ were finite, then we would have been in this situation to start with, by taking any $\phi$ and $\mathcal L = \mathcal K$. Hence, from now on we may ignore the case distinction of whether $\mathcal K$ is infinite.
		
		Take ${\mathbb L} \in \mathcal L$. Note that each $H \in \mathcal H_{G_0}$ is commensurable to some subgroup in $\mathcal H$, so $\betti{H}{i}{\K} = 0$ for each $\K \in \mathcal K$ and $i \leqslant m$. By \cref{prop:tuple_sigma}, there is a finite-index normal subgroup $G_1$ of $G_0$ and an open set $V \subseteq S(G_1)$ such that 
		\[
			\homol_i (G_1; \nov {\mathbb L} {G_1} {\psi}) = \homol_j (H; \nov {\mathbb L} {H} {\psi|_H}) = 0
		\]
		for all $i \leqslant n$, $j \leqslant m$, $\psi \in V$, and $H \in \mathcal H_{G_1}$. Moreover, $\mathring{\overline V} \supseteq S(G_0)$ and $V$ is invariant under the conjugation action of $G_0$ on $S(G_1)$ and under the antipodal map.
		
		Since vanishing of Novikov homology is an open property for characters, there exists an antipodally symmetric open set $W$ in $S(G_1)$ containing $\phi\vert_{G_1}$ such that for every $\psi \in W$ we have
		\[
			\homol_i (G_1; \nov \K {G_1} \psi) =  \homol_j (H; \nov \K H {\psi\vert_H}) = 0 
		\]
		for every $i\leqslant n$, $j\leqslant m$, $H \in \mathcal H_{G_1}$, and $\K \in \mathcal K \s- \mathcal L$. We easily arrange for $W$ to be invariant under the action of $G_0$ on $S(G_1)$, since this action factors through the finite group $G_0/G_1$; we also easily arrange for $W$ to be invariant under the antipodal map. Fix an integral character $\psi \in W \cap V$. We obtain
		\[
			\homol_i (G_1; \nov {{\mathbb L}} {G_1} { \pm \psi}) = \homol_j (H; \nov {{\mathbb L}} H {\pm \psi\vert_H}) = 0 
		\]
		for every $i\leqslant n$, $j\leqslant m$, and $H \in \mathcal H_{G_1}$; we also have $\psi$ not vanishing on the groups $H$. Hence, by passing to a deeper finite-index subgroup $G_1$, replacing $\mathcal H_{G_0}$ by a potentially more numerous, but still finite, tuple $\mathcal H_{G_1}$, and replacing $\phi$ by $\psi$, we have guaranteed that all the relevant kernels are of type $\typeFP_n(\K)$ or $\typeFP_m(\K)$ for all fields $\K \in \mathcal K \smallsetminus \mathcal L$, and for the field ${\mathbb L}$. Since $\mathcal L$ is finite, repeating this process finitely many times yields the desired finite-index subgroup and homomorphism.
		\qedhere
	\end{enumerate}
\end{proof}

The arguments above also show the following: if $G$ is a RFRS group of type $\typeFP$ with $\operatorname{cd}(G) = n$ and $\betti G n \K = 0$ for all fields $\K$ in a small set of fields $\mathcal K$, then there is a finite-index subgroup $G_0 \leqslant G$ and an epimorphism $\phi \colon G_0 \rightarrow \Z$ such that $\homol^n(G_0;\nov{\K}{G_0}{\pm\phi}) = 0$ for all $\K \in \mathcal K$. This uses \cref{cor:large_primes}(\ref{item:large_primes_cohomol}) (see also \cite{Fisher2024_cdKernels}*{Proposition 3.2}). Applying \cite{Fisher2024_cdKernels}*{Theorem 3.5} yields the following corollary.

\begin{cor}
	Let $G$ be a RFRS group of type $\typeFP$ and suppose that $\operatorname{cd}(G) = n$.
	\begin{enumerate}
		\item If $\betti{G}{n}{\K} = 0$ for all fields $\K$ in some small set of fields $\mathcal K$, then there exists a subgroup $G_0 \leqslant G$ of finite index and an epimorphism $\phi \colon G_0 \rightarrow \Z$ such that $\operatorname{cd}_\K(\ker \phi) = n-1$ for all $\K \in \mathcal K$.
		\item If $\betti{G}{n}{\K} = 0$ for all prime fields $\K$, then there exists a subgroup $G_0 \leqslant G$ of finite index and an epimorphism $\phi \colon G_0 \rightarrow \Z$ such that $\operatorname{cd}_\K(\ker \phi) = n-1$ for all fields $\K$.
	\end{enumerate}
\end{cor}

\subsection{Poincar\'e-duality groups}
\label{sec PD fibring}

In this section we will work with orientable Poincar\'e-duality pairs $(G, \mathcal H)$. Recall that we allow $\mathcal H = \varnothing$, in which case by our convention we are dealing with an orientable Poincar\'e duality group $G$. Again, we will insist on $R$ being a commutative ring in this subsection.

\begin{lem}
	\label{pd half fibring}
	Let $(G,\mathcal H)$ be an orientable $\PD_R^n$-pair, and let $\phi \colon G \to \Z$ be a non-trivial homomorphism. If $\ker \phi$ is of type $\typeFP_{\lfloor n/2 \rfloor}(R)$, and $\ker \phi|_H$ is a proper
	 subgroup of $H$ of type $\typeFP_{\lfloor (n-1)/2 \rfloor}(R)$ for each $H \in \mathcal H$, then 
	\[
	\homol_i(G; \nov R G {\pm \phi}) = \homol^i(G; \nov R G {\pm \phi}) = 0
	\]
	for all $i$.
\end{lem}
\begin{proof}
	By \cref{Sikorav hom and cohom}, the fact that $\ker \phi$ is of type $\typeFP_{\lfloor n/2 \rfloor}(R)$ is equivalent to
	\[
	\homol^i(G; \nov R G {\pm \phi}) = \homol_i(G; \nov R G {\pm \phi}) = 0
	\]
	for all $i \leqslant \lfloor n/2 \rfloor$. Poincar\'e duality of $(G,\mathcal H)$ gives
	\[
	\homol_i(G, \mathcal H; \nov R G {\pm \phi}) = 0
	\]
	for all $i \geqslant n - \lfloor n/2 \rfloor = \lceil n/2 \rceil$, see \cref{PD}.
	
	At this point we split the proof into two cases. First, suppose that $\mathcal H = \varnothing$. Since the homology of the pair $(G,\mathcal H)$ coincides with the homology of $G$ in this case, we have established that
	\[
	\homol_i(G; \nov R G {\pm \phi}) = 0
	\] 
	for all $i$. The cohomology statement follows from \cref{flip}, or from \cref{PD}.
	
	Now suppose that $\mathcal H \neq \varnothing$. Since every $H \in\mathcal H$ is itself an orientable $\PD^{n-1}_R$-group by \cref{peripheral PD}, we may apply the previous case and conclude that
		\[
	\homol_i(H; \nov R H {\pm \phi|_H}) = 0
	\]
	for every $H \in \mathcal H$ and every $i$  (we are using the fact that $\phi|_H$ is non-trivial here).

	We claim that $\homol_i(H; \nov{R}{G}{\pm\phi}) = 0$ for all $H \in \mathcal H$ and every $i$ as well. Indeed, the vanishing of $\homol_i(H; \nov R H {\pm \phi|_H})$ for all $i$ is witnessed by chain contractions, which can be viewed as matrices with entries in $\nov R H {\pm \phi|_H}$. But $\nov R H {\pm \phi|_H} \leqslant \nov{R}{G}{\pm\phi}$, so the same chain contractions witness the vanishing of $\homol_i(H; \nov{R}{G}{\pm\phi})$ for all $i$. 
	
	The long exact homology sequence for the pair $(G,\mathcal H)$ now tells us that 	
	\[
	\homol_i(G; \nov R G {\pm \phi}) \cong \homol_i(G, \mathcal H; \nov R G {\pm \phi})
	\]
	for all $i$, and hence both are zero.
The vanishing of Novikov cohomology follows from \cref{flip}.
\end{proof}

 When $\mathcal H = \varnothing$ and $R = \Z$, the following result is a  special case of a theorem of Hillman--Kochloukova~\cite{HillmanKochloukova2007}*{Corollary 6.1}; see also \cite{Bieri1981}*{Theorem 9.11}. Both Hillman--Kochloukova and Bieri use spectral sequences, whereas our proof avoids that, making use of Novikov rings instead.

\begin{thm}
	\label{HK} 
	Let $(G,\mathcal H)$ be an orientable Poincar\'e-duality pair over $R$ in dimension $n$ with a non-trivial homomorphism $\phi \colon G \to \Z$ that does not vanish on any of the subgroups in $\mathcal H$. If $\ker \phi$ is of type $\typeFP_{\lfloor n/2 \rfloor}(R)$ and $\ker \phi|_H$ is of type $\typeFP_{\lfloor (n-1)/2 \rfloor}(R)$ for every $H \in \mathcal H$, then $(\ker \phi, \mathcal H_\phi)$ is an orientable Poincar\'e-duality pair over $R$ in dimension $n-1$.
\end{thm}
\begin{proof}
	The combination of \cref{pd half fibring}, Poincar\'e duality, and \cref{les}  immediately gives an isomorphism of $R \ker \phi$-modules
	\[
	\homol^i(\ker \phi, \mathcal H_\phi; R \ker \phi) \cong \homol^{i+1}(G, \mathcal H; R G) \cong \left\{ \begin{array}{cl} R & \textrm{ if } i = n-1, \\ 0 & \textrm{ otherwise.} \end{array} \right.
	\]

Combining the vanishing of Novikov homology that we  obtained in \cref{pd half fibring} with \cref{Sikorav} immediately yields that $\ker \phi$ is of type $\typeFP_\infty(R)$. 
 Since $G$ is of finite cohomological dimension over $R$, so is $\ker \phi$, and therefore $\ker \phi$ is of type $\typeFP(R)$ by  a suitable adaptation of \cite{Brown1982}*{Proposition VIII.6.1} from $\Z$ to $R$.
 
 We split the proof into two cases. First, if $\mathcal H = \varnothing$,  then the above facts tell us immediately that $\ker \phi$ is an orientable Poincar\'e-duality group over $R$ in dimension $n-1$.
 
 Now let us suppose that $\mathcal H \neq \varnothing$. Since every $H \in \mathcal H$ is an orientable $\PD^{n-1}_R$-group by \cref{peripheral PD}, we may apply the previous case and conclude that $\mathcal H_\phi$ consists solely of orientable $\PD^{n-2}_R$-groups. Hence for every $H' \in \mathcal H_\phi$ we have 
 \[
 	\homol^i(H'; R \ker \phi) \cong \bigoplus_{\ker \phi / H'} \homol^i(H'; R H')\cong \left\{ \begin{array}{cl} R \ker \phi / H' & \textrm{ if } i = n-2, \\ 0 & \textrm{ otherwise.} \end{array} \right.
 \]
 The long exact cohomology sequence for the pair $(\ker \phi,\mathcal H_\phi)$ now tells us that
 \[
  \homol^i(\ker \phi; R \ker \phi) \cong \left\{ \begin{array}{cl}  \Delta_{\ker \phi / \mathcal H_\phi} & \textrm{ if } i = n-2, \\ 0 & \textrm{ otherwise,} \end{array} \right.
 \]
 as required -- note that we had to change a direct product over $\mathcal H_\phi$ to a direct sum. This is possible since $\mathcal H_\phi$ is finite, as $\phi$ does not vanish on any of the groups in $\mathcal H$, and as $ \mathcal H$ is finite by \cref{peripheral PD}.
\end{proof}

Hillman--Kochloukova offer a similar result \cite{HillmanKochloukova2007}*{Corollary 6.1} when $R = \Z$ and $\mathcal H = \varnothing$, where they only require $\ker \phi$ to be of type $\typeFP_{\lfloor (n-1)/2 \rfloor}(R)$, and additionally ask for the Euler characteristic of $G$ to be zero. In fact, it  is not difficult to see that the additional assumption $\chi(G) = 0$ is equivalent to requiring that $\homol_{\lfloor m/2 \rfloor}(G;\nov{\Z}{G}{\pm\phi}) = 0$ using the fact that $\nov{\Z}{G}{\pm\phi}$ is von Neumann finite by \cite{Kochloukova2006} (see the proof of \cite{HillmanKochloukova2007}*{Theorem 6}). Note that we will only use the stated version of \cref{HK} in our proof of \cref{thm:PD_pair_fibring} below.

\begin{thm}\label{thm:PD_pair_fibring}
	Let $\mathcal K$ be a small set of fields.
	Let $(G,\mathcal H)$ be an orientable $\PD^n_{\mathcal K}$-pair, and suppose that $G$ is RFRS. The following are equivalent:
	\begin{enumerate}
		\item\label{item:poincare_fibre_pairs} There exists a finite-index subgroup $G_0 \leqslant G$ and an epimorphism $\phi \colon G_0 \to \Z$ that does not vanish on any of the groups in $\mathcal H_{G_0}$ such that $(\ker \phi,  (\mathcal  H_{G_0})_\phi)$  is an orientable $\PD^{n-1}_\K$-pair for every $\K \in \mathcal K$.
		\item\label{item:vanish_modp_pairs} We have $\betti{G}{i}{\K} = 0$ for all $i$ and all $\K \in \mathcal K$. 
		\item\label{item:weak_vanishing_pairs} We have $\betti{G}{i}{\K} = 0$ for all $i \leqslant \lfloor n/2 \rfloor$ and all $\K \in \mathcal K$, and  $\betti{H}{i}{\K} = 0$ for all $i\leqslant \lfloor (n-1)/2 \rfloor$, $\K \in \mathcal K$, and $H \in \mathcal H$.
	\end{enumerate}
\end{thm}
\begin{proof}
	We will prove a cycle of implications.
	\begin{enumerate}
	\item[\eqref{item:poincare_fibre_pairs} $\Rightarrow$ \eqref{item:vanish_modp_pairs}] In this case, $G$  virtually algebraically fibres with kernel of type $\typeFP(\K)$ for all $\K \in \mathcal K$, so item (\ref{item:vanish_modp_pairs}) follows by \cref{mapping torus,betti nrs scaling}.
		
		\item[\eqref{item:vanish_modp_pairs} $\Rightarrow$\eqref{item:weak_vanishing_pairs}] The assumption implies that $\homol^i(G,\mathcal H;\D_{\K G}) = 0$ for all $i$ by Poincar\'e duality. The vanishing of homology of $G$ with coefficients in $\D_{\K G}$ gives the vanishing of cohomology, for example using chain contractions. The long exact sequence of \cref{les for pairs} now gives vanishing of the cohomology of $H$ with coefficients in $\D_{\K G}$. Since $\D_{\K G}$ contains $\D_{\K H}$, and since skew-fields are flat over one another, this gives vanishing of the cohomology of $H$ with coefficients in $\D_{\K H}$. Poincar\'e duality of $H$ now gives		
		 $\betti{H}{i}{\K} = 0$ for all $i$ and all $\K$, showing that item (\ref{item:weak_vanishing_pairs}) holds.
		
		\item[\eqref{item:weak_vanishing_pairs}$\Rightarrow$ \eqref{item:poincare_fibre_pairs}] 
		By \cref{peripheral PD}, the tuple $\mathcal H$ is finite, and it consists of groups of type $\typeFP(\mathcal K)$.
		\cref{l2 all chars} gives us a finite-index subgroup $G_0$ and an epimorphism $\phi \colon G_0 \to \Z$ that does not vanish on any of the subgroups in $\mathcal H_{G_0}$, such that $\ker \phi$ is of type $\typeFP_{\lfloor n/2 \rfloor}(\K)$, and $\ker \phi|_{H'}$ is of type $\typeFP_{\lfloor (n-1)/2 \rfloor}(\K)$ for every $H' \in \mathcal H_{G_0}$ and every $\K \in \mathcal K$. Also, $(G_0, \mathcal H_{G_0})$ is an orientable $\PD^n_\K$-pair  for every $\K \in \mathcal K$; indeed, the cohomology of $G_0$ with coefficients in $\K G_0$ behaves as required by Shapiro's lemma, and $G_0$ is of type $\typeFP(\K)$, being a finite-index subgroup of $G$. Now \cref{HK}  finishes the argument. \qedhere
	\end{enumerate}
\end{proof}

Taking $\mathcal H = \varnothing$  and $\mathcal K$ to be a single field immediately implies our first main result.

\begin{mainthm}[\ref{main pd single field}]
	Let $\K$ be a field and $G$ be a RFRS orientable Poincar\'e-duality group of dimension $n>0$ over $\K$. The following are equivalent:
	\begin{enumerate}
		\item There exists a finite-index subgroup $G_0$ in $G$ and an epimorphism $\phi \colon G_0 \to \Z$ such that $\ker \phi$ is an orientable Poincar\'e-duality group of dimension $n-1$ over $\K$;
		\item For every $i \leqslant n$ we have $\betti{G}{i}{\K} = 0$.
	\end{enumerate}
\end{mainthm} 

Being a Poincar\'e-duality group over a prime field immediately implies being a Poincar\'e-duality group over all fields of the same characteristic \cite{DicksDunwoody1989book}*{Proposition V.3.7}. 

Taking $\mathcal H = \varnothing$ and $\mathcal K$ to be the set of prime fields yields the version that deals with all fields simultaneously.  

\begin{mainthm}[\ref{main pd}]
	Let $G$ be a RFRS orientable Poincar\'e-duality group of dimension $n>0$. The following are equivalent:
	\begin{enumerate}
		\item There exists a finite-index subgroup $G_0$ in $G$ and an epimorphism $\phi \colon G_0 \to \Z$ such that $\ker \phi$ is an orientable  Poincar\'e-duality group of dimension $n-1$ over all fields;
		\item For all $i \leqslant n$ and all prime fields $\K$, we have $\betti{G}{i}{\K} = 0$.
	\end{enumerate}
\end{mainthm}

\subsection{Low dimensions}

In dimension three, \cref{thm:PD_pair_fibring} can give a topological conclusion.

\begin{thm}
	Let $(G,\mathcal H)$ be a $\PD^3$-pair and suppose that $G$ is RFRS.  The following are equivalent:
\begin{enumerate}
	\item\label{item:PD3_kernel} There exists a finite-index subgroup $G_0 \leqslant G$ such that $G_0$ is the fundamental group of a fibred three-manifold, and $\mathcal H_{G_0}$ is the collection of fundamental groups of the boundary components of the manifold;
	\item\label{item:b1zero} We have $\betti{G}{1}{\Q} =  0$.
\end{enumerate}
\end{thm}
\begin{proof}[Sketch proof]
Fibred three-manifolds have vanishing first $L^2$-Betti numbers, so it is clear that item (\ref{item:PD3_kernel}) implies item (\ref{item:b1zero}). Thus, we focus on the converse implication.

If $\betti G 1 \Q = 0$, then $\homol^2(G,\mathcal H;\D_{\Q G}) = 0$ by Poincar\'e duality. Moreover, the groups $H \in \mathcal H$ are all $\PD^2$ groups. This implies that $\homol^2(H;\D_{\Q G})$ vanishes for all $H \in \mathcal H$. The long exact sequence of \cref{les for pairs} then implies that $\betti G 2 \Q = 0$, since vanishing of $L^2$-homology in a degree is equivalent to the vanishing of the $L^2$-cohomology in the same degree; hence $\betti G i \Q = 0$ for all $i$. By \cref{thm:PD_pair_fibring}, there exists a finite-index subgroup $G_0 \leqslant G$ and an epimorphism $\phi \colon G_0 \rightarrow \Z$ with $\left(\ker \phi, (\mathcal H_{G_0})_\phi\right)$ being an orientable $\PD^2_\Q$-pair.

Every such pair is geometric by \cite{DicksDunwoody1989book}*{Theorem V.10.2} if $\mathcal H \neq \varnothing$. If $\mathcal H = \varnothing$, then by a result of Bowditch \cite{Bowditch2004} we know that $\ker \phi$ is virtually the fundamental group of a closed orientable surface of genus at least one; using the Nielsen realisation theorem, we obtain that $\ker \phi$ itself is  the fundamental group of a closed orientable surface of genus at least one.
 
In either event, $\ker \phi$ is a surface group, and if $\mathcal H \neq \varnothing$ then the surface has boundary and the action of $\Z = \im \phi$ on $\ker \phi$ preserves the peripheral structure which is encoded in $(\mathcal H_{G_0})_\phi$.
 
Orientable surface-by-$\Z$ groups are always realised as fundamental groups of fibred $3$-manifolds \cites{Nielsen1927,Zieschangetal1980}, and therefore $G$ is virtually a $3$-manifold group. The details of this argument can be found in \cite{Bridsonetal2025}.\qedhere
\end{proof}

We also offer the following result in dimension four, where we can conclude that the kernel of the virtual algebraic fibration is in fact a Poincar\'e-duality group over $\Z$.

\begin{thm}
	Let $(G,\mathcal H)$ be a $\PD^4$-pair and suppose that $G$ is RFRS. The following are equivalent:
	\begin{enumerate}
		\item\label{item:PD4_kernel} There exist a finite-index subgroup $G_0 \leqslant G$ and an epimorphism $\phi \colon G_0 \rightarrow \Z$ that does not vanish on any of the groups in $\mathcal H_{G_0}$ such that $(\ker \phi, (\mathcal H_{G_0})_\phi)$ is an orientable $\PD^3$-pair.
		\item\label{item:b1b2zero} We have $\betti{G}{1}{\Q} = \betti{G}{2}{\Q} = 0$.
	\end{enumerate}
\end{thm}
\begin{proof}
	Item \eqref{item:PD4_kernel} implies item \eqref{item:b1b2zero} by \cref{mapping torus,betti nrs scaling}, noting that the group $\ker \phi$ and the groups in $(\mathcal H_{G_0})_\phi$ are all of type $\typeFP$ -- for the latter groups this follows from \cref{peripheral PD}. 
	
	We now focus on the converse implication. By potentially passing to an index-two subgroup, we may assume that $(G, \mathcal H)$ is orientable.
	The group $G$ is of cohomological dimension $3$ or $4$ (depending on whether $\mathcal H \neq \varnothing$), and so it is infinite. Similarly, the groups in $\mathcal H$ are all $\PD^3$-groups by \cref{peripheral PD}, and so they are also of cohomological dimension $3$, and thus infinite. We conclude that $\betti{G}{0}{\Q} = 0$ and $\betti{H}{0}{\Q} = 0$ for every $H \in \mathcal H$, see \cite{Lueck2002}*{Theorem 6.54(8)(b)}.

	We claim that $\betti H 1 \Q = 0$ for all $H \in \mathcal H$. Indeed, the long exact sequence of \cref{les for pairs} gives an isomorphism
	\[
		\prod_{H \in \mathcal H} \homol^1(H;\D_{\Q G}) \cong \homol^2(G,\mathcal H;\D_{\Q G}).
	\]
	But $\homol^2(G,\mathcal H;\D_{\Q G}) \cong \homol_2(G;\D_{\Q G}) = 0$ by Poincar\'e duality and by assumption, proving that $\homol^1(H;\D_{\Q G}) = 0$ for every $H \in \mathcal H$. Since $\D_{\Q G}$ is a flat $\D_{\Q H}$-module, as both are skew-fields, we obtain $\homol^1(H;\D_{\Q H}) = 0$ for every $H \in \mathcal H$. Finally, vanishing of $L^2$-homology and $L^2$-cohomology in a given degree is equivalent, and so $\betti H 1 \Q = 0$ for all $H \in \mathcal H$, as claimed.
	
	By \cref{l2 all chars} with $\mathcal K = \{ \Q\}$, $n=2$, and $m=1$, there is a finite-index subgroup $G_0 \leqslant G$ and an epimorphism $\varphi \colon G_0 \rightarrow \Z$ that does not vanish on any of the groups in $\mathcal H_{G_0}$ such that $\ker \phi$ is of type $\typeFP_2(\Q)$, and $\ker \phi|_{H}$ is of type $\typeFP_1(\Q)$ for every $H\in \mathcal H_{G_0}$.
	
	Being of type $\typeFP_1(\Q)$ is equivalent to being finitely generated, which is in turn the same as  being of type $\typeFP_1$. Applying \cref{Sikorav}  now yields
	\[
		\homol_i(G_0; \nov \Z {G_0} {\pm\phi}) = 0 \quad \text{and} \quad \homol_i(H; \nov \Z H {\pm\phi|_H}) = 0
	\]
	for all $H \in \mathcal H_{G_0}$ and $i \in \{0,1\}$. By \cref{Sikorav,HK}, it is therefore enough to prove that
	\[
		\homol_2(G_0;\nov{\Z}{G_0}{\pm\phi})  = 0.
	\]
	Using Poincar\'e duality, this is equivalent to showing that
	\[
	 \homol^2(G_0,\mathcal H_{G_0}; \nov \Z{G_0}{\pm\phi}) = 0,
	\]
	which we will now do.

	The  portion of the long exact sequences in homology associated to the pair $(G_0,\mathcal H_{G_0})$
	\[
		\homol_1(G_0;\nov{\Z}{G_0}{\pm\phi}) \rightarrow \homol_1(G_0,\mathcal H_{G_0};\nov{\Z}{G_0}{\pm\phi}) \rightarrow \bigoplus_{H \in \mathcal H_{G_0}} \homol_0(H; \nov{\Z}{G_0}{\pm\phi})
	\]
	shows that $\homol_1(G_0,\mathcal H_{G_0};\nov{\Z}{G_0}{\pm\phi}) = 0$. Similarly, the portion
	\[
		\homol_2(G_0;\nov{\Q}{G_0}{\pm\phi}) \rightarrow \homol_2(G_0,\mathcal H_{G_0};\nov{\Q}{G_0}{\pm\phi}) \rightarrow \bigoplus_{H \in \mathcal H_{G_0}} \homol_1(H; \nov{\Q}{G_0}{\pm\phi})
	\]
	shows that $\homol_2(G_0,\mathcal H_{G_0};\nov{\Q}{G_0}{\pm\phi}) = 0$, since $\homol_2(G_0;\nov{\Q}{G_0}{\pm\phi}) = 0$ by \cref{Sikorav}.

	Suppose that $\mathcal H_{G_0} \neq \varnothing$. The pair $(G_0, \mathcal H_{G_0})$ is a $\PD^4$-pair, and hence is of type $\typeFP$, which means precisely that $\Delta_{G_0/\mathcal H_{G_0}}$ defined over $\Z$ admits a finite resolution by finitely generated projective $\Z G_0$-modules.
	Thus, fix a resolution
	\[
		\cdots \rightarrow C_2 \rightarrow C_1 \rightarrow C_0 \rightarrow \Delta_{G_0/\mathcal H_{G_0}} \rightarrow 0
	\]
	of $\Delta_{G_0/\mathcal H_{G_0}}$ by finitely generated free $\Z G_0$-modules (passing from projective to free modules potentially makes this resolution infinite in length). The previous paragraph implies that there is a partial chain contraction
	\[
		s_i \colon \novq{G_0}{\pm\phi} \otimes_{\Z G_0} C_i \rightarrow \novq{G_0}{\pm\phi} \otimes_{\Z G_0} C_{i+1}, \quad i \in \{0,1\},
	\]
	of the chain complex $\novq{G_0}{\pm\phi} \otimes_{\Z G_0} C_\bullet$.  
	 Moreover, viewed as a matrix over $\novq{G_0}{\pm\phi}$, the chain contraction $s_0$ can be chosen to have entries in $\nov{\Z}{G_0}{\pm\phi}$. (Here, we are really carrying out two arguments, one for $\phi$ and one for $-\phi$; we use $\pm \phi$ for convenience.)

	Let $z \colon C_1 \rightarrow \nov{\Z}{G_0}{\pm\phi}$ be a $1$-cocycle, and consider its extension
	\[
		\overline z \colon \novq{G_0}{\pm\phi} \otimes_{\Z G_0} C_1 \rightarrow \novq{G_0}{\pm\phi}.
	\]
	Then $\overline z = (\partial_1 s_0 + s_1 \partial_2) \overline z = \partial_1 s_0 \overline z$.  Restricting both sides to $C_1$, we see that $z$ is the coboundary of the cochain $s_0 z \colon C_0 \rightarrow \nov \Z {G_0} {\pm\phi}$. Hence,
	\[
	\homol^2(G_0,\mathcal H_{G_0}; \nov \Z{G_0}{\pm\phi}) =	\operatorname{Ext}_{\Z G_0}^1(\Delta_{G_0/\mathcal H_{G_0}};\nov{\Z}{G_0}{\pm\phi}) =  0,
	\]
	as desired. 
	
	When $\mathcal H_{G_0} = \varnothing$, the argument is completely analogous with the trivial $\Z G$-module  $\Z$ used instead of $\Delta_{G_0/\mathcal H_{G_0}}$, and with indices of the chain contractions increased by one.
\end{proof}

Restricting to the case $\mathcal H = \varnothing$, we obtain the statement advertised in the introduction.

\begin{maincor}[\ref{main dim 4}]
	Let $G$ be a RFRS Poincar\'e-duality group in dimension four. The following are equivalent:
	\begin{enumerate}
		\item There exists a finite-index subgroup $G_0 \leqslant G$ and an epimorphism $\phi \colon G_0 \rightarrow \Z$ such that $\ker \phi$ is an orientable Poincar\'e-duality group of dimension three;
		\item We have  $\betti{G}{1}{\Q}=\betti{G}{2}{\Q} = 0$.
	\end{enumerate}
\end{maincor}

\bibliography{bibliography}

\end{document}